\documentclass[11pt,times,twoside]{article}
\usepackage{graphicx,latexsym,euscript,makeidx,color,bm}
\usepackage{amsmath,amsfonts,amssymb,amsthm,thmtools,mathtools,mathrsfs,enumerate}
\usepackage[colorlinks,linkcolor=blue,anchorcolor=green,citecolor=red]{hyperref}

\usepackage{geometry}
\def\5n{\negthinspace \negthinspace \negthinspace \negthinspace \negthinspace }
\def\4n{\negthinspace \negthinspace \negthinspace \negthinspace }
\def\3n{\negthinspace \negthinspace \negthinspace }
\def\2n{\negthinspace \negthinspace }
\def\1n{\negthinspace }
     
 \def\sB{\mathscr{B}}  \def\cB{{\cal B}}

\def\dbE{\mathbb{E}}   \def\cE{{\cal E}}  
\def\dbF{\mathbb{F}} \def\sF{\mathscr{F}}  \def\cF{{\cal F}}  
         
   \def\cH{{\cal H}}

   \def\cK{{\cal K}}

\def\dbN{\mathbb{N}}     
     
\def\dbP{\mathbb{P}}     
   \def\cQ{{\cal Q}}  
\def\dbR{\mathbb{R}}     
   \def\cS{{\cal S}}

   \def\cV{{\cal V}}

\def\Om{\Omega}           \def\Th{\Theta}

\def\ss{\smallskip}                
\def\ms{\medskip}                
               
\def\ds{\displaystyle}           
\def\ra{\rightarrow}      
 
\def\no{\noindent}        \def\q{\quad}                      
\def\ns{\noalign{\ss}}    \def\qq{\qquad}                    
    \def\hb{\hbox}

  \def\deq{\triangleq}               
            \def\({\Big (}
                  \def\){\Big )}
\def\leq{\leqslant}       \def\geq{\geqslant}
\def\ges{\geqslant}       \def\esssup{\mathop{\rm esssup}}   \def\[{\Big[}
           \def\]{\Big]}
\def\h{\widehat}                   \def\cd{\cdot}
\def\wt{\widetilde}              \def\cds{\cdots}
                            
        \def\ts{\times}                      \def\pa{\partial}

\def\a{\alpha}        \def\G{\Gamma}   \def\g{\gamma}   \def\O{\Omega}   \def\o{\omega}
\def\b{\beta}         \def\D{\Delta}   \def\d{\delta}   \def\F{\Phi}     
         \def\Th{\Theta}  \def\th{\theta}    \def\si{\sigma}
   \def\L{\Lambda}  \def\l{\lambda}        
    \def\t{\tau}     \def\f{\varphi}  \def\i{\infty}   

\def\bde{\begin{definition}\label}    \def\ede{\end{definition}}
\def\be{\begin{equation}}
\def\bel{\begin{equation}\label}      \def\ee{\end{equation}}
\def\bt{\begin{theorem}\label}        \def\et{\end{theorem}}
\def\bc{\begin{corollary}\label}      \def\ec{\end{corollary}}
\def\bl{\begin{lemma}\label}          \def\el{\end{lemma}}
\def\bp{\begin{proposition}\label}    \def\ep{\end{proposition}}
\def\bas{\begin{assumption}\label}    \def\eas{\end{assumption}}
\def\br{\begin{remark}\label}         \def\er{\end{remark}}
\def\bex{\begin{example}\label}       \def\ex{\end{example}}
\def\ba{\begin{array}}                \def\ea{\end{array}}
\def\ben{\begin{enumerate}}           \def\een{\end{enumerate}}

\newtheorem{theorem}{Theorem}[section]
\newtheorem{definition}[theorem]{Definition}
\newtheorem{proposition}[theorem]{Proposition}
\newtheorem{corollary}[theorem]{Corollary}
\newtheorem{lemma}[theorem]{Lemma}
\newtheorem{remark}[theorem]{Remark}
\newtheorem{example}[theorem]{Example}

\newtheorem{assumption}{Assumption}

\makeatletter
   
   \@addtoreset{equation}{section}
\makeatother

\allowdisplaybreaks[4]

\begin{document}

\title{\bf  A local maximum principle for robust optimal control problems of
quadratic  BSDEs\thanks{This paper is supported by National Key R\&D Program of China (Grant Nos.2022YFA1006101, 2018YFA0703900), National Natural Science Foundation of China (Grant Nos.12171279, 72171133, 12101291, 12371445), Natural Science Foundation of Shandong Province (Grant Nos.ZR2020MA032, ZR2022MA029), Natural Science Foundation of Guangdong Province of China (Grant No.2214050003543), and Science and Technology Commission of Shanghai Municipality (Grant No.22ZR1407600).}}

\author{Tao Hao\thanks{School of Statistics and Mathematics, Shandong University of Finance and Economics, Jinan 250014, China (Email: {\tt taohao@sdufe.edu.cn}).
}~,~~
%
Jiaqiang Wen\thanks{Department of Mathematics, Southern University of Science and Technology, Shenzhen 518055, China (Email: {\tt wenjq@sustech.edu.cn}).  
}~,~~
Qi Zhang\thanks{School of Mathematical Sciences, Fudan University, Shanghai 200433, China \&
Laboratory of Mathematics for Nonlinear Science, Fudan University, Shanghai 200433, China (Email:{\tt qzh@fudan.edu.cn}). 
}~,~~
}

\maketitle

\no\bf Abstract: \rm
 The paper concerns the necessary maximum principle  for robust optimal control problems of quadratic BSDEs.
 The coefficient of the systems depends on the parameter $\th$,
  and the generator of BSDEs is of quadratic growth in $z$. Since the model is uncertain, the variational inequality is proved by weak convergence technique.
 In addition, due to the generator being quadratic with respect to $z$,  the
forward adjoint equations  are SDEs  with unbounded coefficient involving
mean oscillation martingales.
Using   reverse H\"older inequality and   John-Nirenberg inequality, we show that its solutions  are  continuous  with respect to the parameter $\th$.
The necessary and sufficient conditions for robust optimal control are proved by   linearization method.
%
\ms

\no\bf Key words:\rm\  quadratic BSDE, model uncertainty, maximum principle, robust optimal control


\ms

\no\bf AMS subject classifications. \rm  93E20, 60H10, 35K15

\section{Introduction}\label{sec 1}

The stochastic maximum principle, namely, the necessary condition for optimality,
is an important approach  to study stochastic optimal control problems.
In 1990, Peng \cite{Peng-1990} first proved a global maximum principle for classical stochastic control systems.
From then on, this topic  has been extensively explored for various stochastic systems. We refer to
Hu and Tessitore \cite{Fuhrman-Hu-Tessitore-2013} for infinite-dimensional dynamics,
Buckdahn, Li and Ma \cite{Buckdahn-Li-Ma-2016} for mean-field control systems, Fuhrman,
Hu, Ji and Xue  \cite{Hu-Ji-Xue-2018} for fully coupled  forward-backward stochastic control systems,
to name but a few.
In addition, many works on   maximum principle for the stochastic recursive optimal control problems have been published. For example, Peng \cite{Peng-1993}
considered a local maximum principle in the case that the control domain is convex and  the diffusion coefficient depends on control.
Ji and Zhou \cite{Ji-Zhou-2006} investigated a local maximum principle for stochastic optimal control with terminal state constraints.
Xu \cite{Xu-1995} established a stochastic maximum principle for the nonconvex control domain in the case that the   diffusion coefficient
is independent of  control. Hu and Ji \cite{Hu-Ji-2016} studied stochastic maximum principle for stochastic recursive optimal control problems under volatility ambiguity. Hu \cite{Hu-2017} obtained a global stochastic maximum principle by introducing two new adjoint equations, and solved Peng's
open problem completely.


              In many scenarios,  we meet the situation that the model is uncertain and the generator of BSDEs is quadratic growth.
             Our motivation of the present paper is to study the control problem of such kind of model. We give two
examples for the concerned model which come from risk-sensitive control problems and optimal strategy problems for large investor in \autoref{sec111}.

Inspired by the  two examples in \autoref{sec111}, we consider the following robust optimal control problem with quadratic generator.
Let $\th\in \Th$ represent different market conditions, where $\Th$ is a locally compact Polish space with distance $d$.
Assume that $V$ is a given nonempty convex subset of $\dbR^k$, and consider the following forward-backward control system
           \begin{equation}\label{1.1}
             \left\{
              \begin{aligned}
                X^v_\th(t)= & x+\int_0^tb_\th(s,X^v_\th(s),\dbE[X^v_\th(s)],v(s))ds\\
                            &  +\int_0^t\si_\th(s,X^v_\th(s),v(s))dW(s),\ t\in[0,T],\\
                Y^v_\th(t)= & \F_\th(X^v_\th(T), \dbE[X^v_\th(T)])
                              +\int_t^Tf_\th(s,X^v_\th(s),\dbE[X^v_\th(s)], Y^v_\th(s),Z^v_\th(s), v(s) )ds\\
                            &- \int_t^T(Z_\th^v(s))^\top dW_s,\\
              \end{aligned}
             \right.
           \end{equation}
           where $b_\th:[0,T]\ts \dbR^n  \ts \dbR^n \ts V \ra \dbR^n,
           \si_\th:[0,T]  \ts  \dbR^{n}  \ts V \ra \dbR^{n\ts d},$
           $f_\th:[0,T] \ts \dbR^n \ts \dbR^n\ts \dbR  \ts \dbR^d  \ts V \ra \dbR, \
           \F_\th:\dbR^n \ts \dbR^n \ra \dbR$;
          $f_\th$ is quadratic growth with respect to $z$; $v(\cd)$ is a control process;
            the precise assumptions on $b_\th,  \si_\th,  f_\th, \F_\th$  refer to \autoref{ass 1}.
             In order to obtain a general
result, we assume that the forward stochastic differential equation (SDE, for short) is mean-field type, and consider
the following cost functional
\begin{equation*}\label{2.9}
           J(v(\cd))=\sup\limits_{Q\in \cQ}\int_{\Th}   \dbE\[  \phi_\th(X_\th^v(T))+\g_\th(Y^v_\th(0))\] Q(d\th).
           \end{equation*}
           The assumptions on  $\phi_\th$ and $\g_\th$ lie in \autoref{ass 11}.
           The robust optimal control problem with quadratic generator is to minimize the cost function $J(v(\cd))$ over $\cV_{ad}$
           (see \autoref{def 2.1}). The target of
           the present paper is to give the necessary and sufficient conditions of the optimal control.


Note that the solvability of BSDEs with quadratic generators has been intensively investigated over the past twenty years. In 2000, Kobylanski \cite{Kobylanski00} investigated the existence and uniqueness of one-dimensional BSDEs in the case that the generator $g$ is of quadratic growth in $z$ and the terminal value $\xi$ is bounded.
 Briand and Hu \cite{Briand-Hu-06,Briand-Hu-08} obtained the existence and uniqueness of one-dimensional  BSDEs  with quadratic growth and unbounded terminal value.
The multi-dimensional situation with bounded terminal value was studied by Hu and Tang \cite{Hu-Tang-2016} and Xing and Zitkovic \cite{Xing-Zitkovic};
and the multi-dimensional situation with unbounded terminal value was investigated by Fan, Hu and Tang \cite{Fan--Hu--Tang-20}.
Some other recent developments concerning the quadratic BSDEs can be found in
 Fan, Hu and Tang \cite{FHT}, Hu, Li and Wen \cite{Hu-Li-Wen-JDE2021}, Hao, Hu, Tang and Wen   \cite{Hao-Hu-Tang-Wen-2022}  and so on.
In addition, much progress has been made  on  stochastic optimal control problems involving quadratic BSDEs. Let us cite \cite{Lim-Zhou-2005} for its applications to the risk-sensitive control problems with the risk-sensitive parameter; and
\cite{Rouge-El Karoui-2000,Delbaen-Grandits-Rheinlander-Samperi-Schweizer-Stricker-2002, Hu-Imkeller-Muller-2005, Morlais-2009} for  the exponential utility maximization problems.



           Since the generator is quadratic and the model is uncertain, there are three essential difficulties in the
           present work.
           (i) The fact that $f$ is of quadratic growth in $z$ causes that
           the   variational equations  are linear BSDEs with stochastic unbounded Lipschitz coefficients
           involving BMO-martingales. Those results for BSDEs with  bounded Lipschitz coefficients (see e.g. \cite[Lemma 3.2]{Hu-Wang-2020})
           do not work in this setting.
           In order to obtain the maximum principle, we   prove a new estimate (\autoref{pro 3.2}), i.e.,
           for  $p\in(1\vee 2p_{\pa_z f_\th}^{-1},2)$,
         \begin{equation*}
             \begin{aligned}
                \lim\limits_{\l\ra 0}\sup\limits_{\th\in\Th}\dbE\bigg[
             \sup\limits_{0\leq t\leq T}
             \G_\th(t)|\delta Y^\l_\th(t)|^p
             +\(\int_0^{T}(\G_\th(t))^\frac{2}{p}|\delta Z^\l_\th(t)|^2dt\)^{\frac{p}{2}}\bigg]=0,
              \end{aligned}
              \end{equation*}
         where $\G_\th(t):=\cE\(\int_0^t\pa_{z}f_\th(s)^\top dW(s)\),\ t\in[0,T].$
           (ii) The forward adjoint equations are linear SDEs with stochastic unbounded coefficients. We
            obtain the wellposedness of this kinds of equations, which develops the work of  Gal'chuk \cite{Galchuk-1978}.
            The continuity of its solutions with respect to $\th$, i.e.,
            for $1<p<(p_{\pa_z f_\th}\wedge p_{(\pa_z f_{\tilde{\th}}-\pa_z f_{\th})})$,
                \begin{equation*}
                \begin{aligned}
                &    \lim_{\epsilon\ra 0}\sup_{d(\th, \tilde{\th})\leq\epsilon }\dbE\bigg[\sup_{t\in[0,T]}|p_\th(t)-p_{\tilde{\th}}(t)|^p\bigg]=0,
                 \end{aligned}
                \end{equation*}
             is proved by the theory of BMO-martingales (\autoref{le 4.1}),
            which is necessary to prove that $ \L_\th(\cd)$ defined by (\ref{5.4}) is $\cF$-progressively measurable.
            (iii) Since we consider the robust optimal control problems, i.e., the cost functional is a supremum over a set of probability
            measures, the classical convex variational method is not suitable for  this situation. We adopt the weak convergence technique to prove
            the variational inequality.
                %


           Compared with the existing literatures, our paper has the following contributions.
           First, our model is uncertain and the generator $f$ is of quadratic growth in $z$. In this case,  the robust optimal control problems rather than classical optimal control problems are considered.  We prove a new and nontrivial estimate (\autoref{pro 3.2})
           and develop  the work of \cite{Hu-Wang-2020} from the Lipschitz case to quadratic case.
          Second,  we give the  existence and uniqueness result of $L^p$-solution to linear SDEs with stochastic unbounded coefficients,
           and prove the continuity of its solutions with respect to $\th$ by reverse H\"{o}lder inequality and  John-Nirenberg inequality.
           In addition, the continuity of solution to
           linear BSDEs with stochastic unbounded coefficients with respect to  $\th$   is obtained.
           Third,
            since the model is uncertain, the necessary  maximum principe is proved by  linearization  technique and weak convergence method. Moreover, under certain convexity assumptions, with the help of Sion's minimax theorem, we prove the sufficient maximum principle.



           The paper is arranged as follows. In \autoref{sec 2}, we formulate the problem. Variational inequality is proved in \autoref{sec 3}.
           The \autoref{sec 4} is devoted to the necessary and sufficient maximum principles.


\section{Two examples}\label{sec111}

In this section, we present two examples which demonstrate the applications of robust optimal control problems of
quadratic  BSDEs.
\begin{example}(risk-sensitive control)
             Assume that a market is roughly divided into two states $A$ and $B$ (for example,  a share market is  divided into  a bull market or  a bear market) which usually leads to the different coefficients in two states. Let $\th=1$ and 2 denote the state $A$ and the state $B$, respectively, which can be characterized as  $\th\in\Th=\{1,2\}$.  Let
$\cQ=\{Q^\l: \l\in[0,1]\},$
where $Q^\l$ is the  probability  such that $Q^\l(\{1\})=\l,\q  Q^\l(\{2\})=1-\l.$
              For simplicity, we only consider the 1-dimensional case. Suppose that there exist $N$-individual agents in a system. The state process of  the $i$-th agent is described by
                   \begin{equation}\label{7.1111}
                    \left\{
                     \begin{aligned}
                    dX_\th^{i,v}(t) & =b_\th^i(t,X_\th^{i,v}(t), \frac{1}{N}\sum_{j=1}^N X_\th^{j,v}(t),v(t))dt
                                    +\si_\th^i(t,X_\th^{i,v}(t), v(t))dW^i(t),\\
                     X_\th^{i,v}(t)&=x_0,
                    \end{aligned}
                    \right.
                \end{equation}
                where $W^i,i=1,2,\cds,N$ are independent copies of a 1-dimensional standard Brownian motion $W$; $v(\cd)$ is the control process. By $\cV_{ad}$ we denote the set of admissible controls (see \autoref{def 2.1}).
                The robust objective functional of the $i$-th agent is
                     \begin{equation}\label{7.3}
                     \begin{aligned}
                  &J^i(v(\cd)) = \max\limits_{\l\in\{0,1\}}\bigg\{\l Y^{i,v}_1+(1-\l)Y^{i,v}_2\bigg\}
                                    =\sup_{Q^\l\in \cQ}\int_\Th Y^{i,v}_\th Q^\l(d\th),
                                       \end{aligned}
                \end{equation}
                where   $(Y_\th^{i,v}, Z_\th^{i,v} )$ is the solution of the following BSDE with quadratic generator  (see El Karoui and Hamaden\`{e} \cite{El Karoui-Hamadene-2003})
                 \begin{equation}\label{7.3}
                    \left\{
                     \begin{aligned}
                    dY_\th^{i,v}(t) & = -\frac{\kappa}{2}|Z_\th^{i,v}(t)|^2 dt
                                    +Z_\th^{i,v}(t)dW^i(t),\\
                     Y_\th^{i,v}(T)&=\F_\th^i(X_\th^{i,v}(T),\frac{1}{N}\sum_{j=1}^N X_\th^{j,v}(T) ).
                    \end{aligned}
                    \right.
                \end{equation}
                Note that  the parameter $\kappa$ is called risk-sensitive parameter. The  $i$-th agent wants to minimize her/his objective functional.
                The above  problem is regarded as  a risk-sensitive robust optimal control problem for particle systems.

                Now assume that  the game is symmetric. In other words, let
                $b^i_\th=b_\th, \si^i_\th=\si_\th, \F^i_\th=\F_\th.$  As $N\ra \i$ in
                (\ref{7.1111})-(\ref{7.3}), from  strong law of large number, we obtain the following FBSDE:
                   \begin{equation*}\label{7.5}
                    \left\{
                     \begin{aligned}
                    dX_\th^{v}(t) & =b_\th(t,X_\th^{v}(t), \dbE[X_\th^{v}(t)],v(t))dt
                                    +\si_\th(t,X_\th^{v}(t), v(t))dW(t),\\
                     dY_\th^{v}(t) & = -\frac{\kappa}{2}|Z_\th^{v}(t)|^2 dt
                                    +Z_\th^{v}(t)dW(t),\\
                     X_\th^{v}(t)&=x_0,\  Y_\th^{v}(T)=\F_\th(X_\th^{v}(T),\dbE[ X_\th^{v}(T)] ),
                    \end{aligned}
                    \right.
                \end{equation*}
                and, meanwhile, the cost functional becomes  $J(v(\cd))= \sup\limits_{Q^\l\in \cQ}\int_\Th Y_\th^{v}(0)Q^\l(d\th).$
                     Consequently, the control problem can be written as below.

        \textbf{Problem R}  Find an optimal control $\bar{v}(\cd)$ such that
         \begin{equation*}\label{7.6}
           J(\bar{v}(\cd))=\inf_{v(\cd)\in\cV_{ad}}J(v(\cd))=\inf_{v(\cd)\in\cV_{ad}}\sup_{Q^\l\in \cQ}\int_\Th Y_\th^{v}(0)Q^\l(d\th).
         \end{equation*}

\end{example}

\begin{example}(optimal strategy for large investor)
   Let $K\subset\dbR$ be  a closed convex cone.
               Assume  that there are $N$  investors in a market and
              the market consists of a bond and a stock whose prices denoted by $S_{j}(t), j=0,1$ are governed by the following ODE and SDE:
            \begin{equation*}
             \left\{
             \begin{aligned}
             dS_{0}(t)&=S_{0}(t)\[r(t)+f^i_{0}(t,x^{i,\pi}(t), \frac{1}{N}\sum_{j=1}^N x^{j,\pi}(t),\pi(t))\] dt,\\
             dS_{1}(t)&=S_{1}(t)\[(\mu(t)-f^i_{1}(t,x^{i,\pi}(t), \frac{1}{N}\sum_{j=1}^N x^{j,\pi}(t),\pi(t))    dt+\si dW(t)\],
            \end{aligned}
            \right.
             \end{equation*}
            where $(r,\mu) : \O \ts [0,T] \ra \dbR\ts \dbR$  are the interest rate and return rate, respectively; $\si$ is the volatility of
            stock; $(r, \mu)$ are $\{\cF_t\}_{t\geq0}$-adapted and uniformly bounded processes and $\si\neq 0$;
 $W$ is a 1-dimensional Brownian motion; for each $i=1,2,\cds,N$, $f^i_{\ell},\ell=0,1: [0,T]\ts\dbR^+\ts \dbR^+ \ts \dbR\ra\dbR$ are two given functions representing the effect of the strategies chosen by  the investors on the prices.

            The $i$-th investor intends to invest $\pi^{i}(t)$ in stock at time $t$ and the remaining asset in bond.  Under a self-financed portfolio, the wealth process of the $i$-th investor, starting with an initial wealth $x^i_{0}$, satisfies the following wealth equation:
          \begin{equation*}\label{1.2111}
          \left\{
           \begin{aligned}
            dx^{i,\pi}(t)&=\(r(t)x^{i,\pi}(t)+( x^{i,\pi}(t)-\pi^i(t)) f^i_{0}(t, x^{i,\pi}(t), \frac{1}{N}\sum_{j=1}^N x^{j,\pi}(t),\pi(t))   \\
            &\q \ +\pi^{i}(t)[\mu(t)-r(t)+ f^i_{1}(t, x^{i,\pi}(t), \frac{1}{N}\sum_{j=1}^N x^{j,\pi}(t),\pi(t))]   \)dt+\pi^{i}(t)\si dW(t),\q t\in[0,T],\\
            x^{i,\pi}(0)&=x^i_{0}.
            \end{aligned}
          \right.
         \end{equation*}
        For $i=1,2,\cds,N$, let $\g^i>0$ be  a given parameter and $C^i$ be the consumption at time T for the $i$-investor. Consider the following cost functional

         \begin{equation*}
         J^i(\pi)=\frac{1}{\g^i}\ln\dbE\[\exp[-\g^i(x^{i,\pi}(T)-C^i)]\].
         \end{equation*}

        A strategy $\pi$ taking values in $K$ is called admissible if for any $t\in[0,T]$, $\dbE\[\int_0^T|\pi(t)|^2dt\]< \i.$ By $\cK$ we denote the
         set of all admissible strategies.
         The $i$-th investor intends to minimize her/his cost functional over $\cK$.
         Next, let $N\ra\i$ and consider the asymptotic behavior of $N$ investors. For simplicity, we set $\pi=\pi^{i}, x_0=x_0^i, \g=\g^i,
         f_{0}=f^i_{0}, f_{1}=f^i_{1}$ and $C=C^i$. This
         optimal strategy problem can be summarized  as

         \textbf{Problem}  Find an optimal strategy $\bar{\pi}(\cd)$ such that
         \begin{equation*}
         J(\bar{\pi}(\cd))= \inf\limits_{\pi(\cd)\in\cK}\frac{1}{\g}\ln\dbE\[\exp[-\g(x^{\pi}(T)-C)]\],
         \end{equation*}
         subject to a mean-field SDE
                  \begin{equation*}\label{1.21112}
          \left\{
           \begin{aligned}
            dx^{\pi}(t)&=\(r(t)x^{\pi}(t)+( x^{\pi}(t)-\pi(t)) f_{0}(t,x^{\pi}(t), \dbE[ x^{\pi}(t)],\pi(t))   \\
            &\q\ \ +\pi(t)[\mu(t)-r(t)+ f_{1}(t, x^{\pi}(t), \dbE[ x^{\pi}(t)],\pi(t))]   \)dt+\pi(t)\si dW(t),\q t\in[0,T],\\
            x^{\pi}(0)&=x_0.
            \end{aligned}
          \right.
         \end{equation*}
         To illustrate  our problem in a concise way, set
         $$f_{0}(t,x^{\pi}(t), \dbE[ x^{\pi}(t)],\pi(t)):=\a(t),\q  f_{1}(t,x^{\pi}(t), \dbE[ x^{\pi}(t)],\pi(t)):=\b(t).$$ Here $\a(\cd)$ and $\b(\cd)$ are two bounded functions.
        In order to compute the forward value of a portfolio, let us consider a zero-coupon bond $\L$ with a maturity of $T$, which
         is a financial asset generating a cash-flow of $1$ at time $T$. In other words, there is an $\dbR$-valued progressively
          measurable process $\G$  such that the zero-coupon price $\L$ satisfies BSDE:
          \begin{equation*}\label{1.2112}
          \left\{
           \begin{aligned}
            d\L(t)&=\L(t)\[\(r(t)+\G(t)(\mu(t)-r(t)+\b(t)-\a(t))\)dt
            +\G(t)\si dW(t)\],\q t\in[0,T],\\
            \L(T)&=1.
            \end{aligned}
          \right.
         \end{equation*}
         Here, for simplicity, we assume that $\G$ is $\{\cF_t\}_{t\geq0}$-adapted and uniformly bounded for any $(t,\o)\in[0,T]\ts \O.$
         According to \cite[Theorem 2.1]{Rouge-El Karoui-2000}, one can see
         \begin{equation}\label{1.8}
         J(\bar{\pi})= \inf\limits_{\pi(\cd)\in\cK}\frac{1}{\g}\ln\dbE\[\exp[-\g(x^{\pi}(T)-C)]\]=-\frac{x_0}{\L(0)}+\sup_{v(\cd)\in\cV}Y^v(0),
         \end{equation}
         where $(Y^v, Z^v)$ satisfies BSDE
         \begin{equation}\label{1.9}
         \left\{
         \begin{aligned}
         dY^v(t)&=-\(-\frac{1}{2\g}|\delta(t)+\Pi_{\si^{-1}\tilde{K}}(-\delta(t)-
         \g Z^v(t))|^2-(\delta(t)+\Pi_{\si^{-1}\tilde{K}}(-\delta(t)-\g Z^v(t)))Z^v(t)\)  dt\\
         &\q
         +Z^v(t)dW(t),\ \ \  t\in[0,T],\\
         Y^v(T)&=C,
         \end{aligned}
         \right.
         \end{equation}
         where $\delta(t)=\si^{-1}(\mu(t)-r(t)+\b(t)-\a(t)-\si^2\G(t))$; $\widetilde{K}=\{x\in \dbR|\sup\limits_{\pi\in K}(-\pi x)<\i\}$;
         $\Pi_{\si^{-1}\tilde{K}}(u)$ is the projection of $u\in\dbR$ on  $\si^{-1}\tilde{K}$;
        $\cV:=\{v:\O\ts [0,T]\ra\si^{-1}\wt{K}|\ v\  \text{is}\ \{\cF_t\}_{t\geq0}\text{-adapted}\  \text{and}$\   $\text{bounded}.$
          Since $|\Pi_{\si^{-1}\tilde{K}}(u)|\leq |u|,$ for $\forall u\in \dbR$, there exists a constant $L>0$ such that the generator
          of (\ref{1.9}) is dominated by  $L(1+|Z^v(t)|^2)$.

         The cost functional (\ref{1.8}) tells us that  solving an optimal strategy problem is equivalent to solving an optimal control problem of  BSDE,
         whose generator is dominated by a quadratic generator.
         Following this line, a natural question is if $\g$ and $C$ depend on a parameter $\th$ and are continuous in $\th$, how to describe the equivalent
         optimal control problem,  where $\th\in\Th$ represents the different market states and  $\Th$  is a closed set of $\dbR$.
         Actually, in this case, we need to consider
         the following cost functional
         \begin{equation*}\label{5.12}
         \widehat{J}(v(\cd)):=\sup_{Q\in \cQ}\int_\Th Y^v_\th(0) Q(d\th),
         \end{equation*}
         where $\cQ$ is a weakly compact and convex set of probability measures on $(\Th, \cB(\Th))$;
           $(Y^v_\th, Z^v_\th)$ satisfies the following BSDE
         \begin{equation}\label{5.13}
         \left\{
         \begin{aligned}
         dY^v_\th(t)&=-\(-\frac{1}{2\g_\th}|\delta(t)+\Pi_{\si^{-1}\tilde{K}}(-\delta(t)-
         \g_\th Z^v(t))|^2-(\delta(t)+\Pi_{\si^{-1}\tilde{K}}(-\delta(t)-\g_\th Z^v(t)))Z^v(t) \)  dt\\
         &\q\ +Z^v_\th(t)dW(t),\ t\in[0,T],\\
         Y^v_\th(T)&=C_\th.
         \end{aligned}
         \right.
         \end{equation}
         Similar to (\ref{1.9}), the generator 
         of (\ref{5.13}) is dominated by $L(1+|Z^v_\th(t)|^2)$, where $L$ is a constant independent of $\th$.
         Consequently, the corresponding optimal control problem is formulated as follow.

                    \textbf{Problem S}  Find an optimal control  $\bar{v}(\cd)$ such that
         \begin{equation*}
         \widehat{J}(\bar{v}(\cd))=\sup_{v(\cd)\in\cV} \sup_{Q\in \cQ}\int_\Th Y^v_\th(0) Q(d\th).
         \end{equation*}
         \end{example}


\section{Problem formulation}\label{sec 2}

\subsection{Some properties for BMO-martingales}
             Denote by $\dbN$ the set of  all natural numbers and by $\dbR^+$ the set of positive real numbers, respectively.
             Let the superscript $\top$ denote the transpose of vectors or matrices.
             Let $M=(M_t,\sF_t)$ be a uniformly integrable martingale with $M_0=0$, and we set, for $p\geq1$,
$$\|M\|_{BMO_p(\dbP)}:=
\sup_\t\bigg\|\dbE_\t\Big[\( \langle M\rangle_{\i} -\langle M\rangle_{\t} \)^{\frac{p}{2}}\Big]^{\frac{1}{p}}\bigg\|_{\i},$$
where $\t $  is a stopping time on $[0,T]$.
The class $\big\{M: \|M\|_{BMO_p(\dbP)}<\i\big\}$ is denoted by $BMO_p(\dbP)$. Note that $BMO_p(\dbP)$ is a Banach space under the norm $\| \cd \|_{BMO_p(\dbP)}$. For simplicity,  $BMO_2(\dbP)$ is written as  $BMO$.

 Next, we list some properties for $BMO$-martingales. The reader can  refer to Kazamaki \cite{Kazamaki-1994} for more details.

$\bullet$ Denote by $\cE(M)$ the Dol\'{e}ans--Dade exponential of a continuous local martingale  $M$, i.e.,
  $\cE(M_t) = \exp\{M_t-\frac{1}{2} \langle M\rangle_t\}$,    for any $t \in  [0, T]$.
    If $M \in  BMO$,     then $\cE(M)$ is a uniformly integrable martingale.

$\bullet$  Let $\Psi$ be the monotonically decreasing function defined on $(1,\i)$ by
               \begin{equation*}
                \begin{aligned}
                \Psi(p)=\bigg(1+p^{-2}\ln \frac{2p-1}{2(p-1)} \bigg)^\frac{1}{2}-1
                 \end{aligned}
                 \end{equation*}
                and then $\Psi(+\i):=\lim_{x\ra+\i}\Psi(x)=0$. For $M\in BMO$, we can find a positive
                constant $p_M$ which satisfies $\Psi(p_M)=\|M\|_{BMO}$. In particular, set
                $p_M=+\i$ if $\|M\|_{BMO}=0$. Then $p_M$ is uniquely determined.

   $\bullet$    The reverse H\"older inequality:          If $p \in  (1, p_M)$, for any stopping time $\tau  \in  [0, T]$,
                \begin{equation*}
                \dbE\[(\cE(M_T))^p/(\cE(M_\t))^p|\cF_\t    \]\leq K(p,||M||_{BMO}),\  \text{a.s.},
                \end{equation*}
                where
                \begin{equation*}
                K(p, ||M||_{BMO})
                =2\bigg(1-\frac{2p-2}{2p-1}\exp\bigg\{p^2\[||M||^2_{BMO}+2||M||_{BMO}    \]    \bigg\}\bigg)^{-1}.
                \end{equation*}

                $\bullet$  The John-Nirenberg inequality: For $M\in BMO$, if $\theta\in(0, \|M\|^{-2}_{BMO}),$ for any stopping time $\tau  \in  [0, T]$,
                \begin{equation*}
                \dbE\[\exp\Big\{ \theta(\langle M\rangle_T- \langle M\rangle_{\t}   )     \Big\}    \Big|\cF_\t\]\leq
                \(1-\theta\|M\|_{BMO}^2\)^{-1},\  \text{a.s.}
                \end{equation*}

    $\bullet$                 By $p^*_M$  we denote the conjugate exponent of $p_M$, i.e.,
             $(p_M)^{-1}+(p^*_M)^{-1}=1$.

\ms

           For any $p\ges1$, $t\in[0,T]$ and a filtration $\dbF$, we introduce some useful spaces:
\begin{align*}
%
%
%
%
%
\ns\ds \cS_\dbF^p(t,T;\dbR^m)=\Big\{&\f:\Om\ts[t,T]\to\dbR^m\bigm|\f\hb{ is
$\dbF$-adapted, continuous, }\\
\ns\ds&\qq\qq\qq
\|\f\|_{\cS_\dbF^p(t,T)}\deq\Big\{\dbE\(\sup_{s\in[t,T]}|\f_s|^p\)\Big\}^{\frac{1}{p}}<\i\Big\},\\
\ns\ds \cS_\dbF^\infty(t,T;\dbR^m)=\Big\{&\f:\Om\ts[t,T]\to\dbR^m\bigm|\f\hb{ is
$\dbF$-adapted, continuous,  }\\
\ns\ds&\qq\qq\qq
\|\f\|_{\cS_\dbF^\infty(t,T)}\deq\esssup_{(s,\o)\in[t,T]\times\Om}|\f_s(\o)|<\i\Big\},\\
\ns\ds \cH_\dbF^{2,p}(t,T;\dbR^m)=\Big\{& \f:\Om\ts[t,T]\to\dbR^m\bigm|\f\hb{ is
$\dbF$-progressively measurable, }\\
\ns\ds&\qq\qq\qq
\|\f\|_{\cH^{2,p}_\dbF(t,T)}\deq\dbE\[\(\int^T_t|\f_s|^2ds\)^{p\over2}\]^{\frac{1}{p}}<\i\Big\}.
%
%
\end{align*}
%

Next, we formulate the optimal control problem. For this end, we first introduce the set of admissible controls.

\begin{definition}\label{def 2.1}\rm
          An  $\dbF$-adapted process $v$ taking values in $V\subset\mathbb{R}^K$ is called an admissible control, if for any $\ell>0$,
          $\dbE[\int_0^T|v(t)|^\ell dt ]<\i.$  By $\cV_{ad}$ we denote the set of all admissible controls.
\end{definition}

\ms

The following assumptions on   coefficients $b_\th,\si_\th, f_\th, \F_\th$ are forced.

         \bas{ass 1}\rm
            \begin{enumerate}[~~\,\rm (i)]
              \item  There exists a constant $C_0>0$ independent  of  $\th$ such that, for $t\in[0,T]$, $x_1, x_2, x'_1, x'_2\in \dbR^n, v, v'\in V$,
                     \begin{equation*}
                     \begin{aligned}
                     &|b_\th(t,0, 0, v)|+|\si_\th(t,0, v)|\leq C_0(1+|v|),\\
                     &|b_\th(t,x_1, x'_1, v)-b_\th(t,x_2,x'_2, v')|\leq C_0(|x_1-x_2|+|x'_1-x'_2|+|v-v'|),\\
                     &|\si_\th(t,x_1,v)-\si_\th(t,x_2, v')|\leq C_0(|x_1-x_2|+|v-v'|);\\
                     \end{aligned}
                     \end{equation*}
              $b_\th,\si_\th$ is continuously differentiable in $(x,x',v)$ and $(x,v)$, respectively;
                     $\pa_x b_\th, \pa_{x'}b_\th, \pa_v b_\th,$ $\pa_x \si_\th,   \pa_v \si_\th $ are Lipschitz continuous in $(x, x',v)$.

              \item  $\Phi_\th,\pa_x\Phi_\th,\pa_{x'}\Phi_{\th}$ are continuous in $(x,x')$ and bounded.
              \item  $f$ is continuously differentiable in $(x,x',y,z,v)$;
              $\pa_xf_\th,\pa_{x'}f_\th,\pa_yf_\th,\pa_zf_\th, \pa_vf_\th $ are Lipschitz continuous in $(x,x',y,z,v)$;
              there exists a constant $C_1>0$  independent  of  $\th$ such that, for $t\in[0,T]$, $x_1, x_2, x'_1, x'_2\in \dbR^n, y_1, y_2\in \dbR,  z_1, z_2\in \dbR^d, v\in V$,
              \begin{equation*}
              \begin{aligned}
              & |f_\th(t,x_1,x'_1,0,0,v)|\leq C_1,\\
              & |f_\th(t,x_1,x'_1, y_1,z_1, v)-f_\th(t,x_2,x'_2,y_2,z_2,v)|\\
              &\leq C_1(|x_1-x_2|+ |x'_1-x'_2|+|y_1-y_2|)
              +C_1(1+|z_1|+|z_2|) |z_1-z_2|,\\
              %
              %
              & |\pa_vf_\th(t,x_1,x'_1,y_1,z_1, v)|
               \leq C_1.
              \end{aligned}
              \end{equation*}
              \item  There exists a constant $C_2>0$  independent  of  $\th$  such that, for any $t\in[0,T]$, $\th,\bar{\th}\in \Th$, $x_1,x'_1\in \dbR^n, y_1\in \dbR,  z_1\in \dbR^d, v\in V$,
              \begin{equation*}
              |\f_{\th}(t,x_1,x'_1, y_1,z_1,v)-\f_{\bar{\th}}(t,x_1,x'_1,y_1,z_1,v)|\leq C_2d(\th, \bar{\th}),
              \end{equation*}
              where $\f_{\th}$ is $b_\th, \si_\th,f_\th, \F_\th$ and their derivatives w.r.t. their respective variables. 
              \item $\cQ$ is a weakly compact and convex set of probability measures on $(\Th, \cB(\Th))$.
            \end{enumerate}
         \eas

            According \cite[Proposition 3]{Briand-Hu-08} and \cite[Lemma 2.1]{Hu-Tang-2016}, we know  the existence and  uniqueness of
           $(Y^v_\th,Z^v_\th).$
          \begin{theorem}\label{th 2.3}\rm
          Under  \autoref{ass 1}, for any $v(\cd)\in \cV_{ad}$ and $p>1$, the equation (\ref{1.1}) has a unique solution
          $(X^v_\th, Y^v_\th, Z^v_\th)\in \cS^p_{\dbF}(0,T;\dbR^n) \ts \cS^\i_{\dbF}(0,T;\dbR) \ts \cH^{2,p}_\dbF(0,T;\dbR^d)$. Moreover, the
          following estimates hold:
          \begin{equation}\label{2.1-0}
          \begin{aligned}
           & \mathrm{(i)}\ \|X^v_\th \|^p_{\cS^p_{\dbF}(0,T)}\leq C\bigg\{|x|^p+\dbE\[\(\int_0^T|b_\th(s,0,0,v(s))|ds\)^p \]
          +\dbE\[\(\int_0^T|\si_\th(s,0,v(s))|^2ds\)^\frac{p}{2} \]\bigg\},\\
          & \mathrm{(ii)}\  \|Y^v_\th\|_{\cS_\dbF^\infty(0,T)}\leq  M_1  \q~ \text{and}\q~  \|Z^v_\th\cdot W\|_{BMO}\leq M_2,
          \end{aligned}
          \end{equation}
          where the constant $C$ depends on $C_0, p,T$ and the constants $M_1, M_2$ depends on $ C_0, C_1, T, \|\Phi\|_\i$.
          \end{theorem}

 In the rest of this paper we use $C$ to represent a generic constant which only depends on given parameters and could be different from line to line.

       The continuity of $(X_\th^v,Y_\th^v,Z_\th^v)$ with respect to $\th$ is proved in the following proposition.
\begin{proposition}\label{pro 2.3}\rm
Under  \autoref{ass 1}, for any $v(\cd)\in \cV_{ad}$, $p>1$ and $q>p_{\pa_z f_\th}^{*}$,

              \begin{equation*}
              \begin{aligned}
              & \mathrm{(i)}\q \lim_{\epsilon\ra0}\sup_{d(\th,\tilde{\th})\leq\epsilon}
              \dbE\bigg[\sup_{0 \leq t\leq T} |X_\th^v(t)-X_{\tilde{\th}}^v(t)|^p\bigg]=0,\\
              & \mathrm{(ii)}\q \lim_{\epsilon\ra0}\sup_{d(\th,\tilde{\th})\leq\epsilon}
              \dbE\bigg[\sup_{0 \leq t\leq T} |Y_\th^v(t)-Y_{\tilde{\th}}^v(t)|^q
              +\(\int_0^T|Z_\th^v(t)-Z_{\tilde{\th}}^v(t)|^2dt\)^{\frac{q}{2}} \bigg]=0.
              \end{aligned}
             \end{equation*}
\end{proposition}

          \begin{proof}
          As for $\mathrm{(i)}$,
         denote $\D X(s)=X_\th^v(s)-X_{\tilde{\th}}^v(s),$ and  for $h=b,\si,$
         \begin{equation*}
         \D h(s):=h_\th(s,X_\th^v(s), \dbE[X_\th^v(s)],v(s))-h_{\tilde{\th}}(s,X_\th^v(s), \dbE[X_\th^v(s)],v(s)).
         \end{equation*}
         Then it follows that
         \begin{equation*}
         \begin{aligned}
         \D X(t)&=\int_0^t \(\D b(s)+b_{\tilde{\th}}(s, X_{\tilde{\th}}^v(s)+\D X(s),\dbE[X_{\tilde{\th}}^v(s)]+\dbE[\D X(s)],v(s))\\
                &\q \ \ \ \ \ \ \ -b_{\tilde{\th}}(s, X_{\tilde{\th}}^v(s),\dbE[X_{\tilde{\th}}^v(s)],v(s))\)ds\\
                &\q+\int_0^t \(\D \si(s)+\si_{\tilde{\th}}(s, X_{\tilde{\th}}^v(s)+\D X(s),v(s)) -\si_{\tilde{\th}}(s, X_{\tilde{\th}}^v(s),v(s))\)dW(s).
         \end{aligned}
         \end{equation*}
         Thanks to \autoref{th 2.3}, it yields, for $p>1$,
           \begin{equation*}
         \begin{aligned}
         \dbE\[\sup\limits_{t\in[0,T]} |\D X(t)|^p\]\leq C \dbE\[\(\int_0^T |\D b(s)|ds\)^p+\(\int_0^T |\D \si(s)|^2ds\)^\frac{p}{2}\]
         \leq C d(\th, \tilde{\th})^p.
         \end{aligned}
         \end{equation*}

         Then we turn to $\mathrm{(ii)}$.  For simplicity, we denote
          \begin{equation*}
         \begin{aligned}
          \D Y(s)&=Y_\th^v(s)-Y_{\tilde{\th}}^v(s),\q \D Z(s)=Z_\th^v(s)-Z_{\tilde{\th}}^v(s),\\
          \D  \Phi(T)&=\Phi_\th(X_{\tilde{\th}}^v(T), \dbE[X_{\tilde{\th}}^v(T)])- \Phi_{\tilde{\th}}(X_{\tilde{\th}}^v(T), \dbE[X_{\tilde{\th}}^v(T)]),\\
          \D f(s)&=f_\th(s, X_{\tilde{\th}}^v(s),\dbE[ X_{\tilde{\th}}^v(s)],Y_{\tilde{\th}}^v(s),Z_{\tilde{\th}}^v(s),v(s))-f_{\tilde{\th}}(s, X_{\tilde{\th}}^v(s),\dbE[ X_{\tilde{\th}}^v(s)],Y_{\tilde{\th}}^v(s),Z_{\tilde{\th}}^v(s),v(s)).
         \end{aligned}
         \end{equation*}

         Then
         \begin{equation*}
         \begin{aligned}
          \D Y(t)&=\D  \Phi(T)+\Phi_\th(X_\th^v(T),\dbE[X_\th^v(T)])-\Phi_\th(X_{\tilde{\th}}^v(T),\dbE[X_{\tilde{\th}}^v(T)])\\
                  &\q+\int_t^T\(\Delta f(s)
                  + f_{\th}(s, X_\th^v(s),\dbE[ X_\th^v(s)],Y_{\tilde{\th}}^v(s)+\D Y(s),Z_{\tilde{\th}}^v(s)+\D Z(s),v(s))\\
                  &\q\ \ \ \ \ \ \ \ \ \ - f_{\th}(s, X_{\tilde{\th}}^v(s),\dbE[ X_{\tilde{\th}}^v(s)],Y_{\tilde{\th}}^v(s),Z_{\tilde{\th}}^v(s),v(s))\)ds\\
                  &\q -\int_t^T \D Z(s)dW(s), t\in[0,T].
         \end{aligned}
         \end{equation*}

         According to \cite[Corollary 9]{Briand-Confortola-2008}, we know, for any $q>p_{\pa_z f_\th}^{*}$, 
           \begin{equation*}
         \begin{aligned}
         & \dbE\[\sup\limits_{t\in[0,T]}|\D Y(t)|^q+\(\int_0^T|\D Z(t)|^2dt\)^\frac{q}{2} \]\\
&\leq C\bigg(  \dbE\[|\D \Phi(T)|^{q+1}+|\D X(T)|^{q+1}+ |\dbE[\D X(T)]|^{q+1}  \]\\
         &\q\ +\dbE\[\int_0^T\(|\D f(t)|+|f_{\th}(t, X_\th^v(t),\dbE[ X_\th^v(t)],Y_{\tilde{\th}}^v(t),Z_{\tilde{\th}}^v(t),v(t))\\
                  &\q\ \ \ \ \ \ \ \ \ \ \ \ \ \ \ - f_{\th}(t, X_{\tilde{\th}}^v(t),\dbE[ X_{\tilde{\th}}^v(t)],Y_{\tilde{\th}}^v(t),Z_{\tilde{\th}}^v(t),v(t))|\)^{q+1}dt\]
         \bigg)^ \frac{q}{q+1}.
          \end{aligned}
         \end{equation*}
         Hence, from (i), we have, for  $q>p_{\pa_z f_\th}^{*}$,
          \begin{equation*}
         \begin{aligned}
           \dbE\[\sup\limits_{t\in[0,T]}|\D Y(t)|^q+\(\int_0^T|\D Z(t)|^2dt\)^\frac{q}{2} \] \leq Cd(\th,\tilde{\th})^q.
          \end{aligned}
         \end{equation*}
\end{proof}

 \ms

          In order to introduce the cost functional, we need the following assumption for the mappings $\phi_\th:\dbR^n\ra \dbR$ and $\g_\th:\dbR\ra \dbR$.

         \bas{ass 11}\rm
            \begin{enumerate}[~~\,\rm (i)]
              \item $\phi_\theta$ and $\g_\th$ are continuously differentiable in their respective variables and bounded.
              \item There exists a constant $L_1>0$  such that, for $x,x'\in \dbR$, $\th\in \Th$,
              \begin{align}\nonumber
               \begin{aligned}
               |\phi_\th(x)-\phi_\th(x')|\leq L_1(1+|x|+|x'|)|x-x'|,\q  |\pa_x\phi_\th(x)-\pa_x\phi_\th(x')|\leq L_1|x-x'|.
               \end{aligned}
              \end{align}
              \item $\g_\th$ and its derivative $\pa_y \g_\th$ are Lipschitz continuous with respect to $y$ uniformly in $\th$ and bounded.
               \item   There exists a constant $L_2>0$  such that, for $x\in \dbR^n$, $\th,\th'\in \Th$,
               \begin{align}\nonumber
                |h_\th(x)-h_{\th'}(x)| \leq L_2d(\th,\th'),
              \end{align}
              where $h$ denotes $\phi, \g$ and their derivatives w.r.t. their respective variables.
            \end{enumerate}
         \eas

\ms

           Since  the control system (\ref{1.1}) is model uncertainty, we consider the robust cost functional
           \begin{equation}\label{2.9}
           J(v(\cd))=\sup\limits_{Q\in \cQ}\int_{\Th}   \dbE\[  \phi_\th(X_\th^v(T))+\g_\th(Y^v_\th(0))\] Q(d\th).
           \end{equation}

           The optimal control problem is

           \textbf{Problem}\ \textbf{(PQU)} Find an optimal control $\bar{v}(\cd)$ such that
           \begin{equation}\label{2.10}
           J(\bar {v}(\cd))=\inf_{v(\cd)\in \cV_{ad}} J(v(\cd)),
           \end{equation}
           subject to (\ref{1.1}) and (\ref{2.9}).

The above control problem is what we concern in the rest of the paper, for which we will give the necessary and sufficient conditions of the optimal control. 

\ms

           \section{Variational inequality}\label{sec 3}

             In this section, we investigate the variational equations and variational inequality, which are essential materials to study the
              stochastic maximum principle.

               Let $\bar {v}(\cd)$  be an optimal control and  $(\bar{X}_\th,\bar{Y}_\th,\bar{Z}_\th)$  be the corresponding state
            processes of (\ref{1.1}). Notice that $\cV_{ad}$ is convex, and hence for any $v(\cd)\in \cV_{ad}$ and $0<\lambda<1$,
            $v^\l(\cd):=\bar{v}(\cd)+\l(v(\cd)-\bar{v}(\cd))\in \cV_{ad}$. By $(X_\th^\l,Y_\th^\l,Z_\th^\l)$ we denote the
            corresponding trajectories of  (\ref{1.1}) with  $v^\l(\cd)$ for each $\th\in \Th$.
            To avoid heavy notation, set
            \begin{equation*}
            \begin{aligned}
            b_\th(t)&=b_\th(t, \bar{X}_\th(t), \dbE[\bar{X}_\th(t)],\bar{v}(t)),\\
            f_\th(t)&=f_\th(t, \bar{X}_\th(t),\dbE[\bar{X}_\th(t)], \bar{Y}_\th(t),\bar{Z}_\th(t), \bar{v}(t)),
            \end{aligned}
            \end{equation*}
            $\si_\th(t), \partial_xb_\th(t),\partial_{x'}b_\th(t), \partial_{v}b_\th(t), \partial_x\si_\th(t),\partial_{x'}\si_\th(t), \partial_{v}\si_\th(t), \partial_xf_\th(t),\partial_{x'}f_\th(t),  \partial_yf_\th(t),\partial_{z}f_\th(t), \partial_{v}f_\th(t)$ can be set similarly.

              \subsection{Variational equation}

            Consider  the following   variational SDE on $[0,T]$: for each $\th\in\Th$,
             \begin{equation}\label{3.1}
             \left\{
             \begin{aligned}
             d X^1_\th(t)  &   =\(\pa_xb_\th(t)X^1_\th(t)+\pa_{x'}b_\th(t)  \dbE[X^1_\th(t)]+\pa_v b_\th(t)(v(t)-\bar{v}(t))\)dt\\
                           & \q  +\sum_{i=1}^d\(\pa_x\si^i_\th(t)X^1_\th(t)
                                   +\pa_v \si^i_\th(t)(v(t)-\bar{v}(t))\)dW^i(t),\\
               X^1_\th(0)  &=0.
             \end{aligned}
             \right.
             \end{equation}
             The equation (\ref{3.1}) is a linear mean-field SDE. Since the coefficients
             $\pa_xb_\th(t),\pa_{x'}b_\th(t),\pa_u b_\th(t), $ $\pa_x\si_\th(t),\pa_{x'}\si_\th(t),\pa_u \si_\th(t)  $ are bounded,
             according to \autoref{th 2.3},  SDE (\ref{3.1}) has a unique solution  $X^1_\th\in\cS^p_{\dbF}(0,T;\dbR^n)$ for $p>1.$
             Furthermore, we have, for $p>1$,
             \begin{equation}\label{3.2}
             \dbE\bigg[ \sup_{0\leq t\leq T} |X_\th^1(t)|^p   \bigg]
             \leq C\bigg\{\dbE\Big[\int_0^T(|v(t)|^p+|\bar{v}(t)|^p)dt    \Big]
             +\Big(\dbE\int_0^T|v(t)|^2dt  \Big)^\frac{p}{2}
             +\Big(\dbE\int_0^T|\bar{v}(t)|^2dt  \Big)^\frac{p}{2}
             \bigg\}.
             \end{equation}

             The following lemma shows that $X^1_\th$ is continuous in $\th$.

             \begin{proposition}\label{pro 3.1-1}\rm
             Under \autoref{ass 1}, we have for $p>1$,
              \begin{equation*}
              \lim_{\epsilon\ra0}\sup_{d(\th,\tilde{\th})\leq\epsilon}
              \dbE\bigg[\sup_{0 \leq t\leq T} |X_\th^1(t)-X_{\tilde{\th}}^1(t)|^p\bigg]=0.
             \end{equation*}
             \end{proposition}

             \begin{proof}
             Notice
             \begin{equation*}\label{3.1111}
             \left\{
             \begin{aligned}
             d \(X^1_\th(t)-X^1_{\widetilde{\th}}(t)\)  &   = \pa_xb_\th(t) (X^1_\th(t)-X^1_{\widetilde{\th}}(t))
                                                              +\pa_{x'}b_\th(t) \dbE[X^1_\th(t)-X^1_{\widetilde{\th}}(t)]
                                                              +I_{1,\th,\wt{\th}}(t)dt\\
                                                        &\  \ \ +  \sum\limits_{i=1}^d\(\pa_x\si^i_\th(t) (X^1_\th(t)-X^1_{\widetilde{\th}}(t))
                                                        +I^i_{2,\th,\wt{\th}}(t)\)dW^i(t),\\
                X^1_\th(t)-X^1_{\widetilde{\th}}(t)   &=0,
            \end{aligned}
             \right.
             \end{equation*}
             where
             \begin{equation*}\label{3.1111-1}
             \begin{aligned}
               I_{1,\th,\wt{\th}}(t)&= \(\pa_xb_\th(t)- \pa_xb_{\wt{\th}}(t)\)   X^1_{\wt{\th}}(t)
                                     +\(\pa_{x'}b_\th(t)- \pa_{x'}b_{\wt{\th}}(t)\) \dbE[X^1_{\wt{\th}}(t)]\\
                                   &\q \ + \(\pa_{v}b_\th(t)- \pa_{v}b_{\wt{\th}}(t)\)\(v(t)-\bar{v}(t)\),\\
              I^i_{2,\th,\wt{\th}}(t)&= \(\pa_x\si^i_\th(t)- \pa_x\si^i_{\wt{\th}}(t)\)   X^1_{\wt{\th}}(t)
                                     + \(\pa_{v}\si^i_\th(t)- \pa_{v}\si^i_{\wt{\th}}(t)\)\(v(t)-\bar{v}(t)\).
            \end{aligned}
             \end{equation*}
             Then we can prove \autoref{pro 3.1-1} by \autoref{th 2.3} and   Buckholder-Davis-Gundy inequality, following a standard proof.
             \end{proof}
             Set $\delta X^\l_\th(t)=\frac{1}{\l}(X_\th^\l(t)-\bar{X}_\th(t))-X^1_\th(t)$, and we have the following estimate on it.

             \begin{lemma}\label{le 3.1}\rm
             Under \autoref{ass 1},  there exists a constant $C>0$ depending on  $C_0$ and $T$ such that,
             for each $\th\in \Th$ and $p>1$,
             \begin{equation*}
             \begin{aligned}
             &\mathrm{(i)}\ \dbE\bigg[\sup\limits_{0\leq t\leq T} |\delta X^\l_\th(t)|^p  \bigg]
                    \leq C\bigg\{\dbE\Big[\int_0^T(|v(t)|^p+|\bar{v}(t)|^p)dt    \Big]+\Big(\dbE\int_0^T|v(t)|^2dt  \Big)^\frac{p}{2}
                                                 +\Big(\dbE\int_0^T|\bar{v}(t)|^2dt  \Big)^\frac{p}{2}\bigg\}.\\
             &\mathrm{(ii)}\  \lim\limits_{\l\ra 0}\sup\limits_{\th\in\Th}\dbE\bigg[
             \sup\limits_{0\leq t\leq T}|\delta X^\l_\th(t)|^p \bigg]=0.
             \end{aligned}
             \end{equation*}
             \end{lemma}
            \begin{proof}
          Set $b^\l_\th(s)=b_\th(s,X^\l_\th(s),\dbE[X^\l_\th(s)],v^\l(s)),
              \si^{i,\l}_\th(s)=\si^i_\th(s,X^\l_\th(s),v^\l(s)).$
              According to the definition of $\delta X^\l_\th$, we have
              \begin{equation}\label{3.3}
              \begin{aligned}
              \delta X^\l_\th(t)   &   =\int_0^t\bigg\{\frac{1}{\l}\(b_\th^\l(s)-b_\th(s)\)-\[ \pa_xb_\th(s)X^1_\th(s)
                                        +\pa_{x'}b_\th(s) \dbE[X^1_\th(s)] +\pa_v b_\th(s)(v(s)-\bar{v}(s))\]\bigg\}ds\\
                                   &   +\int_0^t\sum\limits_{i=1}^d\bigg\{\frac{1}{\l}\(\si^{\l,i}_\th(s)-\si^i_\th(s)\)
                                        -\[ \pa_x\si^i_\th(s)X^1_\th(s)
                                     +\pa_v \si^i_\th(s)(v(s)-\bar{v}(s))\]\bigg\}dW^i(s).
              \end{aligned}
              \end{equation}
              Denote
              \begin{equation*}
              \begin{aligned}
              \pi_\th^{\rho\l}(t) &  =(\bar{X}_\th(t)+ \rho\l(\delta X^\l_\th(t)+X^1_\th(t)),
                                      \dbE[\bar{X}_\th(t)]+ \rho\l(\dbE[\delta X^\l_\th(t)]+\dbE[X^1_\th(t)]),\bar{v}(t)+\rho\l(v(t)-\bar{v}(t))),\\
               \tilde{\pi}_\th^{\rho\l}(t) &  =(\bar{X}_\th(t)+ \rho\l(\delta X^\l_\th(t)+X^1_\th(t)), \bar{v}(t)+\rho\l(v(t)-\bar{v}(t))),\\
                  A_\th^{\l}(t) & =\int_0^1\partial_xb_\th(t, \pi_\th^{\rho\l}(t))d\rho,\qq
                  B_\th^{\l}(t)=\int_0^1\partial_{x'}b_\th(t, \pi_\th^{\rho\l}(t))d\rho,\\
                 C_\th^{\l}(t) & =\int_0^1[\pa_v b_\th(t,\pi_\th^{\rho\l}(t))-\pa_v b_\th(t)](v(t)-\bar{v}(t))d\rho\\
                                   &\q  +[A_\th^{\l}(t)-\pa_xb_\th(t)]X_\th^1(t)
                                      +[B_\th^{\l}(t)-\pa_{x'}b_\th(t)]\dbE[X_\th^1(t)],\\
                  D_\th^{\l,i}(t)& =\int_0^1\partial_x\si^i_\th(t, \tilde{\pi}_\th^{\rho\l}(t))d\rho,\\
                  %
                  F_\th^{\l,i}(t) & =\int_0^1[\pa_v \si^i_\th(t,\tilde{\pi}_\th^{\rho\l}(t))-\pa_v \si^i_\th(t)](v(t)-\bar{v}(t))d\rho  +[D_\th^{\l,i}(t)-\pa_x\si^i_\th(t)]X_\th^1(t).
               \end{aligned}
               \end{equation*}
              Then the equation (\ref{3.3}) can be written as
              \begin{equation*}
              \begin{aligned}
               \delta X^\l_\th(t)   &  =\int_0^t\(A_\th^{\l}(s) \delta X^\l_\th(s)+B_\th^{\l}(s) \dbE[\delta X^\l_\th(s)]
                                       + C_\th^{\l}(s)\)ds\\
                                    &\q  +\sum_{i=1}^d\int_0^t\(D_\th^{\l,i}(s) \delta X^\l_\th(s)
                                     + F_\th^{\l,i}(s)\)dW^i(s).
              \end{aligned}
               \end{equation*}
              Notice that $\pa_xb_\th(t),\pa_{x'}b_\th(t),\pa_vb_\th(t), \pa_x\si^i_\th(t),
              \pa_{v}\si^i_\th(t)$ are bounded.
              Thanks to \autoref{th 2.3} and (\ref{3.2}), it yields, for $p>1$,
                 \begin{equation*}\label{3.4}
              \begin{aligned}
              \dbE\[\sup\limits_{0\leq t\leq T }|\delta X^\l_\th(t)|^p\]
                   &\leq  C\dbE   \bigg(\int_0^T|C_\th^{\lambda}(s)|ds \bigg)^p
                   +C\dbE   \bigg(\int_0^T\sum\limits_{i=1}^d |F_\th^{\lambda,i}(s)|^2ds \bigg)^\frac{p}{2}\\
               &\leq C\bigg\{\dbE\Big[\int_0^T(|v(t)|^p+|\bar{v}(t)|^p)dt    \Big]
                         +\Big(\dbE\int_0^T|v(t)|^2dt  \Big)^\frac{p}{2}
                           +\Big(\dbE\int_0^T|\bar{v}(t)|^2dt  \Big)^\frac{p}{2} \bigg\}.
              \end{aligned}
               \end{equation*}

    \ms

              Next, we turn to prove $\mathrm{(ii)}$.  It is enough to prove
             \begin{equation*}
             \left\{
             \begin{aligned}
              & \lim\limits_{\l\ra0}\sup\limits_{\th\in \Th}  \dbE   \bigg(\int_0^T|C_\th^{\lambda}(s)|ds \bigg)^p=0,\\
             &  \lim\limits_{\l\ra0}\sup\limits_{\th\in \Th}   \dbE\bigg(\int_0^T\sum\limits_{i=1}^d |F_\th^{\lambda,i}(s)|^2ds \bigg)^\frac{p}{2}=0.
             \end{aligned}
             \right.
             \end{equation*}
             On one hand, we have
             \begin{equation*}
             \begin{aligned}
             |C_\th^{\lambda}(s)|^p  &  \leq \(2^{p-1}\vee1\)
                  \bigg\{\int_0^1|\pa_v b_\th(s,\pi_\th^{\rho\l}(s))-\pa_v b_\th(s)|^p|v(s)-\bar{v}(s)|^pd\rho\\
             &\qq\qq\qq\q  +\int_0^1|\pa_x b_\th(s,\pi_\th^{\rho\l}(s))-\pa_x b_\th(s)|^p|X_\th^1(s)|^p d\rho\\
             &\qq\qq\qq\q   +\int_0^1|\pa_{x'} b_\th(s,\pi_\th^{\rho\l}(s))-\pa_{x'} b_\th(s)|^p|\dbE[X_\th^1(s)]|^p d\rho \bigg\}.
             \end{aligned}
             \end{equation*}
              From $ab\leq \frac{1}{2}a^2+\frac{1}{2}b^2, a,b\in \dbR^+$,  we have
              \begin{equation*}
             \begin{aligned}
             |C_\th^{\lambda}(s)|^p  &  \leq C_p\l^p
                  \bigg\{|\delta X^\l_\th(s)+X^1_\th(s)|^{2p}+|\dbE[\delta X^\l_\th(s)+X^1_\th(s)]|^{2p} \\
                  &   \q  +|v(s)-\bar{v}(s)|^{2p}+|X_\th^1(s)|^{2p}+|\dbE[X_\th^1(s)]|^{2p}\bigg\},
             \end{aligned}
             \end{equation*}
             where the constant $C_p$ depends on $p$.
             By (\ref{3.2}) and $\mathrm{(i)}$ in \autoref{le 3.1}
              $\lim\limits_{\l\ra0}\sup\limits_{\th\in \Th}  \dbE   (\int_0^T|C_\th^{\lambda}(s)|ds )^p=0$ follows immediately.
              On the other hand,  based on the fact that
               \begin{equation*}
             \begin{aligned}
             |F_\th^{\lambda,i}(s)|^2&  \leq C\l^2
                  \bigg\{|\delta X^\l_\th(s)+X^1_\th(s)|^{4}+|\dbE[\delta X^\l_\th(s)+X^1_\th(s)]|^{4} \\
                  &   \q     +|v(s)-\bar{v}(s)|^{4}+|X_\th^1(s)|^{4}+|\dbE[X_\th^1(s)]|^{4}\bigg\},
             \end{aligned}
             \end{equation*}
             we obtain
             \begin{equation*}
             \begin{aligned}
             \dbE \bigg(\int_0^T\sum\limits_{i=1}^d |F_\th^{\rho\lambda,i}(s)|^2ds \bigg)^\frac{p}{2}
             &\leq C_{p,d}\l^p
             \bigg\{\dbE\(\int_0^T|v(s)-\bar{v}(s)|^4ds\)^{\frac{p}{2}}\\
             &\q + \dbE\[\sup_{0\leq s \leq T}|\delta X^\l_\th(s)|^{2p}\]
             +\dbE\[\sup_{0\leq s \leq T}|X^1_\th(s)|^{2p}\]\bigg\},
              \end{aligned}
             \end{equation*}
             where the constant $C_{p,d}$ depends on $p$ and $d$.
             Notice that by H\"{o}lder inequality we have\\
            $$\text{if}\  1<p<2, \q \dbE\(\int_0^T|v(s)-\bar{v}(s)|^4ds\)^{\frac{p}{2}}\leq \(\dbE\int_0^T|v(s)-\bar{v}(s)|^4ds\)^{\frac{p}{2}},$$
            and  \\
           $$\text{if}\  p>2,\q \dbE\(\int_0^T|v(s)-\bar{v}(s)|^4ds\)^{\frac{p}{2}}\leq T^{\frac{p}{p-2}}\dbE\int_0^T|v(s)-\bar{v}(s)|^{2p}ds.$$
           Finally, by  (\ref{3.2}) and $\mathrm{(i)}$ in \autoref{le 3.1} again, we derive
           \begin{equation*}
           \lim\limits_{\l\ra0}\sup\limits_{\th\in \Th}   \dbE\bigg(\int_0^T\sum\limits_{i=1}^d |F_\th^{\rho\lambda,i}(s)|^2ds \bigg)^\frac{p}{2}=0.
           \end{equation*}
           \end{proof}

           We next introduce the  variational BSDE on $[0,T]$: for each $\th\in \Th$,
            \begin{equation}\label{3.5}
             \left\{
             \begin{aligned}
             -d Y^1_\th(t)  &   =\bigg(\pa_xf_\th(t)X^1_\th(t)+\pa_{x'}f_\th(t) \dbE[X^1_\th(t)]+\pa_yf_\th(t)Y^1_\th(t)+\pa_zf_\th(t)Z^1_\th(t)\\
             &\q  \ \ \  +\pa_v f_\th(t)(v(t)-\bar{v}(t))\bigg)dt
                                -Z^1_\th(t)dW(t),\\
               Y^1_\th(T)  &=\pa_x\Phi_\th(T)X_\th^1(T)+\pa_{x'}\Phi_\th(T)\dbE[X_\th^1(T)].
             \end{aligned}
             \right.
             \end{equation}
            Since $|\pa_zf_\th(t)|\leq C_2(1+|\bar{Z}_\th(t)|)$, we know from (ii) in \eqref{2.1-0} that $\pa_{z}f_\th \cd W$ is a BMO martingale and the equation (\ref{3.5}) is a linear BSDE with unbounded coefficient.
             Due to $|\pa_x\Phi_\th(T)|\leq C$ and $|\pa_{x'}\Phi_\th(T)|\leq C$,
           by \cite[Proposition 3.1]{Hu-Ji-Xu-2022}, it has a unique solution $(Y,Z)\in \bigcap_{p>1}(\cS^p_{\dbF}(0,T;\dbR) \ts \cH^{2,\frac{p}{2}}_\dbF(0,T;\dbR^d))$, and we further have for all $\bar{p}>p_{\pa_z f_\th}^{*}$ and $p_{\pa_z f_\th}^{*}<p<\bar{p}$,  there exists a constant $C>0$ depending on $p,\bar{p}, T, \|\pa_xf_\th \|_\i, \|\pa_{x'}f_\th \|_\i, \|\pa_yf_\th \|_\i $ and $\|\pa_{z}f_\th \cd W\|_{BMO}$ such that,
             \begin{equation}\label{3.6}
             \begin{aligned}
             & \dbE\bigg[ \sup_{0\leq t\leq T} |Y_\th^1(t)|^p+\bigg( \int_0^{T}  |Z_\th^1(t)|^2dt    \bigg)^\frac{p}{2}   \bigg]\\
             &  \leq C \bigg\{\dbE\Big[\int_0^T(|v(t)|^{\bar{p}}+|\bar{v}(t)|^{\bar{p}})dt    \Big]
               +\Big(\dbE\int_0^T|v(t)|^2dt  \Big)^\frac{\bar{p}}{2}
                                                 +\Big(\dbE\int_0^T|\bar{v}(t)|^2dt  \Big)^\frac{^{\bar{p}}}{2}\bigg\}^\frac{p}{\bar{p}}.
             \end{aligned}
             \end{equation}

             Then we consider the continuity of the pair $(Y^1_\th, Z^1_\th)$ with respect to $\th$.
              \begin{proposition}\label{pro 3.3}\rm
           Under \autoref{ass 1}, for   $q>p_{\pa_z f_\th}^{*} $,
              \begin{equation*}
              \lim_{\epsilon\ra0}\sup_{d(\th,\tilde{\th})\leq\epsilon}
              \dbE\bigg[\sup_{0 \leq t\leq T} |Y_\th^1(t)-Y_{\tilde{\th}}^1(t)|^q+
              \(\int_0^T|Z_\th^1(t)-Z_{\tilde{\th}}^1(t)|^2\)^\frac{q}{2}\bigg]=0.
             \end{equation*}
             \end{proposition}
             The proof is similar to that of \autoref{pro 2.3}.  Hence we omit it to save space.

             \ms

             Set
             \begin{equation}\label{3.9}
             \begin{aligned}
             \delta Y^\l_\th(t)&=\frac{1}{\l}(Y_\th^\l(t)-\bar{Y}_\th(t))-Y^1_\th(t),\q
             \delta Z^\l_\th(t)=\frac{1}{\l}(Z_\th^\l(t)-\bar{Z}_\th(t))-Z^1_\th(t),\\
             \G_\th(t)&:=\cE\(\int_0^t\pa_{z}f_\th(s)^\top dW(s)\),\ \ \ t\in[0,T].
             \end{aligned}
             \end{equation}
             The following proposition demonstrates the continuity of $(\delta Y^\l_\th, \delta Z^\l_\th)$ with respect to $\th$.
              \begin{proposition}\label{pro 3.2}   \rm
             Under \autoref{ass 1}, for  $p\in(1\vee 2p_{\pa_z f_\th}^{-1},2)$,
         \begin{equation*}
             \begin{aligned}
                \lim\limits_{\l\ra 0}\sup\limits_{\th\in\Th}\dbE\bigg[
             \sup\limits_{0\leq t\leq T}
             \G_\th(t)|\delta Y^\l_\th(t)|^p
             +\(\int_0^{T}(\G_\th(t))^\frac{2}{p}|\delta Z^\l_\th(t)|^2dt\)^{\frac{p}{2}}\bigg]=0.
              \end{aligned}
              \end{equation*}

             \end{proposition}
             \begin{proof}
Set
                \begin{equation*}
              \begin{aligned}
              \kappa^{\l \rho}(T)&:=\(\bar{X}_\th(T)+\l\rho(X^1_\th(T)+\d X_\th^\l(T)),
                                      \dbE[\bar{X}_\th(T)]+\l\rho(\dbE[X^1_\th(T)]+\dbE[\d X_\th^\l(T)])\), \\
                \pi_\th^{\l\rho}(t)&:= \(\bar{X}_\th(t)+\l\rho(X^1_\th(t)+\d X_\th^\l(t)),
                                      \dbE[\bar{X}_\th(t)]+\l\rho(\dbE[X^1_\th(t)]+\dbE[\d X_\th^\l(t)])\\
                                   &\q\  \  \ \bar{Y}_\th(t)+\l\rho(Y^1_\th(t)+\d Y_\th^\l(t)),
                                     \bar{Z}_\th(t)+\l\rho(Z^1_\th(t)+\d Z_\th^\l(t)), \bar{v}(t)+\l\rho(v(t)-\bar{v}(t))\),\\
                A^\l_\th(t) & :=\int_0^1\pa_xf_\th(t,\pi_\th^{\l\rho}(t))d\rho,\q
                                    \bar{A}^\l_\th(t):=\int_0^1\pa_{x'}f_\th(t,\pi_\th^{\l\rho}(t))d\rho,\\
                B^\l_\th(t) & :=\int_0^1\pa_yf_\th(t,\pi_\th^{\l\rho}(t))d\rho,\q
                C^\l_\th(t)  :=\int_0^1\pa_zf_\th(t,\pi_\th^{\l\rho}(t))d\rho.
               \end{aligned}
              \end{equation*}  By  the definitions of $\delta Y^\l_\th, \delta Z^\l_\th$, we arrive at
               \begin{equation*}
              \begin{aligned}
                \d Y^{\l}_\th(t) & =  I^{1,\l}_\th+ I^{2,\l}_\th\\
                                   &\q
                                   +\int_t^T \(\pa_xf(s)\d X_\th^{\l}(s)+\pa_{x'}f(s)\dbE[\d X_\th^{\l}(s)]
                                     +\pa_yf(s)\d Y_\th^{\l}(s)+\pa_zf(s)\d Z_\th^{\l}(s)
                                  +D^\l_\th(s) \)ds\\
                                  & \q -\int_t^T\d Z_\th^{\l}(s)dW(s) , \ t\in[0,T],
              \end{aligned}
              \end{equation*}
               where
                \begin{equation*}
              \begin{aligned}
                  I^{1,\l}_\th &:=\int_0^1\pa_x\Phi_{\th}(\kappa^{\l \rho}(T))d\rho   \cd   \d X^\l_{\th}(T)
                                  +\int_0^1\pa_{x'}\F_\th(\kappa^{\l \rho}(T))d\rho   \cd  \dbE[ \d X^\l_\th(T)],\\
                I^{2,\l}_\th &:=  \( \int_0^1\pa_x\F_\th(\kappa^{\l \rho}(T))d\rho-\pa_x\F_\th(T)     \) X_\th^1(T)+   \( \int_0^1\pa_{x'}\F_\th(\kappa^{\l \rho}(T))d\rho-\pa_{x'}\F_\th(T)     \) \dbE[X_\th^1(T)],\\
                D^\l_\th(t) & :=\int_0^1\(\pa_vf_\th(t,\pi_\th^{\l\rho}(t))-\pa_vf_\th(t)\)(v(t)-\bar{v}(t))d\rho+\(A^\l_\th(t)-\pa_xf_\th(t)\)(X^\l_\th(t)-\bar{X}_\th(t))\\
                              &\q\ \
                              + \(\bar{A}^\l_\th(t)-\pa_{x'}f_\th(t)\)\dbE[X^\l_\th(t)-\bar{X}_\th(t)]+\(B^\l_\th(t)-\pa_yf_\th(t)\)(Y^\l_\th(t)-\bar{Y}_\th(t))\\
                              &\q\ \
                              +\( C^\l_\th(t)-\pa_zf_\th(t)   \)(Z^\l_\th(t)-\bar{Z}_\th(t)).
               \end{aligned}
              \end{equation*}
               By \cite[Proposition 3.2]{Hu-Ji-Xu-2022}, it yields, for any $p\in(1\vee 2p_{\pa_z f_\th}^{-1},2)$ and $p<\bar{p}<2$,
               \begin{equation}\label{3.7}
              \begin{aligned}
                & \dbE\bigg[\sup_{t\in[0,T]}\G_\th(t)|\d Y^{\l}_\th(t)  |^p
                +\(\int_{0}^{T}(\G_\th(t))^\frac{2}{p}|\d Z^{\l}_\th(t) |^2dt\)^{\frac{p}{2}} \bigg]\\
                & \leq  C_0\bigg(\dbE\[(\G_\th(T))^\frac{\bar{p}}{p}|I^{1,\l}_\th+I^{2,\l}_\th|^{\bar{p}}
                +\(\int_{0}^T (\G_\th(t))^{\frac{1}{p}}|\pa_xf(t)\d X_\th^{\l}(t)+\pa_{x'}f(t)\dbE[\d X_\th^{\l}(t)]+
                D^\l_\th(t)|dt\)^{\bar{p}}\]\bigg)^{\frac{p}{\bar{p}}}.
             \end{aligned}
              \end{equation}

              First,  reverse H\"{o}lder inequality leads to that, for $p\in(1\vee 2p_{\pa_z f_\th}^{-1},2)$,
              \begin{equation*}
              \dbE[(\G_\th(T))^{\frac{2}{p}}]\leq K(\frac{2}{p}, \|\pa_zf_\th\cd W\|).
\mathcal{}              \end{equation*}
              Consequently, from  \autoref{le 3.1}, we know, for all $p\in(1\vee 2p_{\pa_z f_\th}^{-1},2)$ and $p<\bar{p}<2$,
              \begin{equation*}
              \begin{aligned}
             & \dbE\[(\G_\th(T))^\frac{\bar{p}}{p}|I^{1,\l}_\th|^{\bar{p}}\]\\
              & \leq C\dbE\[(\G_\th(T))^\frac{\bar{p}}{p}\cd(|\d X_\th^\l(T)|^{\bar{p}}+|\dbE[\d X_\th^\l(T)]|^{\bar{p}})\] \\
              & \leq C\bigg\{\dbE\[ (\G_\th(T))^{\frac{2}{p}}\]\bigg\}^{\frac{\bar{p}}{2}}\cd
                      \bigg\{\dbE\[\sup_{t\in[0,T]}|\d X_\th^\l(t)|^{\frac{2\bar{p}}{2-\bar{p}}}\]\bigg\}^{\frac{(2-\bar{p})}{2}}\\
              & \leq C \(K(\frac{2}{p}, \|\pa_zf_\th\cd W\|)\)^{\frac{\bar{p}}{2}}\bigg\{\dbE\[\sup_{t\in[0,T]}|\d X_\th^\l(t)|^{\frac{2\bar{p}}{2-\bar{p}}}\]\bigg\}^{\frac{(2-\bar{p})}{2}}
                      \ra 0,\q \text{as}\ \l\ra0.
             \end{aligned}
              \end{equation*}
  We can similarly deduce that, for all $p\in(1\vee 2p_{\pa_z f_\th}^{-1},2)$ and $p<\bar{p}<2$,
              \begin{equation*}
              \begin{aligned}
              & \lim_{\l\ra 0}\dbE\[(\G_\th(T))^\frac{\bar{p}}{p}|I^{2,\l}_\th|^{\bar{p}}
                +\(\int_{0}^T (\G_\th(t))^{\frac{1}{p}}|\pa_xf(t)\d X_\th^{\l}(t)+\pa_{x'}f(t)\dbE[\d X_\th^{\l}(t)]|dt\)^{\bar{p}}\]=0.
             \end{aligned}
              \end{equation*}

              Next, we focus on  $\dbE\[ \(\int_0^T (\G_\th(t))^{\frac{1}{p}}|D^\l_\th(t)|dt\)^{\bar{p}}\]$.
              Notice
              \begin{equation*}
              \begin{aligned}
              |C^\l_\th(t)-\pa_zf_\th(t)|
              &\leq C\l \( |\d X_\th^\l(t)+X_\th^1(t)|+|\dbE[\d X_\th^\l(t)+X_\th^1(t)]|\\
              &\q   +  |\d Y_\th^\l(t)+Y_\th^1(t)|+|\d Z_\th^\l(t)+Z_\th^1(t)|+|v(t)-\bar{v}(t)|\).
              \end{aligned}
              \end{equation*}
               We  only deal with the most difficult term
                $\dbE\[ \(\int_0^T (\G_\th(t))^{\frac{1}{p}}\cd |\bar{Z}_\th(t)|   \cd |Z_\th^1(t)|dt\)^{\bar{p}}\]$.
                The other terms can be estimated similarly.
                %
                Actually, by H\"{o}lder inequality, we have
                  \begin{equation*}
              \begin{aligned}
                &\dbE\[ \(\int_0^T (\G_\th(t))^{\frac{1}{p}}\cd |\bar{Z}_\th(t)|   \cd |Z_\th^1(t)|dt\)^{\bar{p}}\] \\
                %
                  %
                &\leq \bigg\{\dbE\[\int_0^T (\G_\th(t))^{\frac{2}{p}}\cd |\bar{Z}_\th(t)|^2  dt\]\bigg\}^{\frac{\bar{p}}{2}}
                  \bigg\{\dbE\[\(\int_0^T |Z_\th^1(t)|^2dt\)^\frac{{\bar{p}}}{2-\bar{p}}\bigg]\bigg\}^\frac{2-\bar{p}}{2}.
             \end{aligned}
              \end{equation*}
By H\"{o}lder inequality again, together with Doob's inequality and reverse H\"{o}ler inequality, we further have for
              fixed  $p_1\in(1\vee 2p_{\pa_z f_\th}^{-1},2)$ and any $p>p_1$,

               \begin{equation*}\label{3.11}
               \begin{aligned}
\dbE\[\int_0^T (\G_\th(t))^{\frac{2}{p}}\cd |\bar{Z}_\th(t)|^2  dt\]
               &\leq \bigg\{\dbE\[\sup_{t\in[0,T]}(\G_\th(t))^\frac{2}{p_1}\]\bigg\}^{\frac{p_1}{p}}
               \cd \bigg\{\dbE\[\( \int_0^T|\bar{Z}_\th(t)|^2dt\)^\frac{p}{p-p_1} \]   \bigg\}^{\frac{p-p_1}{p}}     \\
               &\leq \bigg\{\dbE\[(\G_\th(T))^\frac{2}{p_1}\]\bigg\}^{\frac{p_1}{p}}
               \cd \bigg\{\dbE\[\( \int_0^T|\bar{Z}_\th(t)|^2dt\)^\frac{p}{p-p_1} \]   \bigg\}^{\frac{p-p_1}{p}} <\i.
                \end{aligned}
               \end{equation*}
               The above and (\ref{3.6}) implies
               \begin{equation*}
              \begin{aligned}
                &\dbE\[ \(\int_0^T (\G^\l_\th(t))^{\frac{1}{p}}\cd |\bar{Z}_\th(t)|   \cd |Z_\th^1(t)|dt\)^{\bar{p}}\]<\i.
             \end{aligned}
              \end{equation*}
                The proof is complete.
                \end{proof}

               \begin{remark}\label{re 3.3}
                $\mathrm{ (i)}$\ From (\ref{3.7}), we also have
                \begin{equation*}\label{3.7-2}
              \begin{aligned}
                & \dbE\bigg[\sup_{t\in[0,T]}\G^\l_\th(t)|\d Y^{\l}_\th(t)  |^p
                +\(\int_{0}^{T}(\G^\l_\th(t))^\frac{2}{p}|\d Z^{\l}_\th(t) |^2dt\)^{\frac{p}{2}} \bigg]<\i,
             \end{aligned}
              \end{equation*}
              which implies $|\d Y^{\l}_\th(0)|<\i $.

              $\mathrm{ (ii)}$\ Taking $t=0$ it follows from \autoref{pro 3.2}  that,    for $p\in(1\vee 2p_{\pa_z f_\th}^{-1},2)$,
               \begin{equation*}
                 \lim\limits_{\l\ra 0}\sup\limits_{\th\in\Th} |\delta Y^\l_\th(0)|^p =0.
               \end{equation*}
               \end{remark}

               \begin{remark}\label{re 3.3-1}
               If the coefficient $f$ depends on the mean-field terms $\dbE[Y_s]$ and $\dbE[Z_s]$, there is an essential difficulty in the present method to prove the above \autoref{pro 3.2}. We leave it for further work.
               \end{remark}

\subsection{Variational inequality}
                 This subsection is devoted to studying variational inequality.
                  For given $v(\cd)\in\cV_{ad}$, define
                   \begin{equation*}
                   \begin{aligned}
                   \cQ^v  &  =\bigg\{Q\in \cQ \bigg| J(v(\cd))=\int_\Th\dbE\[
                           \phi_\th(X^v_\th(T))+\g_\th(Y_\th^v(0))\] Q(d\th)     \bigg\}.
                   \end{aligned}
                   \end{equation*}

\begin{lemma}\label{le 3.7}\rm
Under \autoref{ass 11}, $\cQ^v$ is nonempty.
\end{lemma}
\begin{proof}
                  From the definition of $J(v(\cd))$, there exists a sequence $Q^n\in \cQ$ such that, for all
                  $v(\cd)\in\cV_{ad}$,
                   \begin{equation*}
                   \int_\Th\dbE\[\phi_\th(X^v_\th(T))+\g_\th(Y_\th^v(0))\] Q^n(d\th)\geq J(v(\cdot))-\frac{1}{n}.
                   \end{equation*}
                  Since $\cQ$ is weakly compact, there is some $Q^v\in \cQ$ such that a subsequence of $Q^n$ (if necessary) converges weakly to $Q^v$. Thanks to \autoref{th 2.3}, \autoref{pro 2.3}  and \autoref{ass 11},
             the mapping $\th\mapsto \dbE[\phi_\th(X^v_\th(T))+\g_\th(Y_\th^v(0))]$ is continuous and bounded. Hence, we obtain
             $$
             \begin{aligned}
J(v(\cd))&\geq\int_{\Th}\dbE\[\phi_\th(X^v_\th(T))+\g_\th(Y_\th^v(0))\] Q^v(d\th)\\
             &=\lim_{n\ra\i}\int_\Th \dbE\[\phi_\th(X^v_\th(T))+\g_\th(Y_\th^v(0))\] Q^n(d\th)\geq J(v(\cd)).
             \end{aligned}
             $$
             Thereby,  $\cQ^v$ is nonempty.
\end{proof}

\ms

                We now give the variational inequality.
                   \begin{theorem}\label{le 3.4}\rm
                    Under \autoref{ass 1} and \autoref{ass 11}, we have
                    \begin{equation*}
             \begin{aligned}
              \lim_{\l\ra0}\frac{1}{\l}\(J(v^\l(\cd))-J(\bar{v}(\cd)) \)
              =\sup_{Q\in \cQ^{\bar{v}}}\int_\Th\dbE\[\pa_x\phi_\th(\bar{X}_\th(T))X^{1}_\th(T) +\pa_y\gamma_\th(\bar{Y}_\th(0)) Y^{1}_\th(0) \]Q(d\th).
             \end{aligned}
             \end{equation*}
              \end{theorem}
               \begin{proof}
                The proof is split into  three steps.

               \ \textbf{Step 1.}  We show
                 \begin{equation}\label{3.13}
             \begin{aligned}
              \liminf_{\l\ra0}\frac{1}{\l}\(J(v^\l(\cd))-J(\bar{v}(\cd)) \)
              \geq\sup_{Q\in \cQ^{\bar{v}}}\int_\Th\dbE\[\pa_x\phi_\th(\bar{X}_\th(T))X^{1}_\th(T) +\pa_y\gamma_\th(\bar{Y}_\th(0)) Y^{1}_\th(0) \]Q(d\th).
             \end{aligned}
             \end{equation}

\ms

               From the definition of  $\cQ^{\bar{v}}$, it yields that,  for any $Q\in\cQ^{\bar{v}}$,
               \begin{equation*}
                \begin{aligned}
                 J(v^\l(\cd))&\geq \int_\Th  \dbE\[\phi_\th(X^\l_\th(T))+\gamma_\th(Y_\th^\l(0))\] Q(d\th),\\
                 J(\bar{v} (\cd))&= \int_\Th  \dbE\[\phi_\th(\bar{X}_\th(T))+\gamma_\th(\bar{Y} _\th(0))\] Q(d\th).
                \end{aligned}
               \end{equation*}
               Consequently,  for any $Q\in\cQ^{\bar{v}}$,
                \begin{equation*}
                \begin{aligned}
                & \frac{1}{\l}\(J(v^\l(\cd))-J(\bar{v} (\cd))\)
                  \geq \int_\Th\dbE\[ \pa_x\phi_\th(\bar{X}_\th(T))X^1_\th(T)+ \pa_y\g_\th(\bar{Y}_\th(0))Y^1_\th(0)\]Q(d\th)
                +J_1^\l+J_2^\l,
                \end{aligned}
                \end{equation*}
                where
                \begin{equation}\label{3.14}
                \begin{aligned}
                 &J_1^\l=\int_\Th \dbE\[ \int_0^1 \pa_x\phi_\th^{\l\rho}(T)d\rho\cd\ \d X_\th^\l(T)
                                          + \int_0^1 \pa_y\g_\th^{\l\rho}(0)d\rho\cd\ \d Y_\th^\l(0)\]Q(d\th),\\
                 &J_2^\l= \int_\Th \dbE\[ \int_0^1 \(\pa_x\phi_\th^{\l\rho}(T)-\pa_x\phi_\th(\bar{X}_\th(T))\)d\rho\cd\  X_\th^1(T)\\
                        &\qq \  + \int_0^1 \(\pa_y\g_\th^{\l\rho}(0)-\pa_y\g_\th(\bar{Y}_\th(0)\) d\rho\cd\ Y_\th^1(0)\]Q(d\th),\\
                 & \pa_x\phi_\th^{\l\rho}(T)= \pa_x \phi_\th(\bar{X}_\th(T)+\l\rho(X^1_\th(T)+\d X^\l_\th(T))),\\
                   &\pa_y\g_\th^{\l\rho}(0)= \pa_y \g_\th(\bar{Y}_\th(0)+\l\rho(Y^1_\th(0)+\d Y^\l_\th(0))).
                \end{aligned}
                \end{equation}

                On the one hand,  since
                $|\int_0^1 \pa_x\phi_\th^{\l\rho}(T)d\rho|\leq  L_1(1+|\bar{X}_\th(T)|+|X^1_\th(T)|+|\d X^\l_\th(T)|)$
                and $|\pa_y\g_\th^{\l\rho}(0)|\leq C$, it follows that, for any $\th\in \Th$,
                 \begin{equation*}
                \begin{aligned}
                 &\dbE\[ \int_0^1 \pa_x\phi_\th^{\l\rho}(T)d\rho\cd\ \d X_\th^\l(T)
                                          + \int_0^1 \pa_y\g_\th^{\l\rho}(0)d\rho\cd\ \d Y_\th^\l(0)\]\\
                 & \leq C  \bigg\{\dbE\[ \sup_{t\in[0,T]}(1+|\bar{X}_\th(t)|^2+|X^1_\th(t)|^2+|\d X^\l_\th(t)|^2)\]  \bigg\}^\frac{1}{2}\cd\
                            \bigg\{\dbE\[\sup_{t\in[0,T]}|\d X_\th^\l(t)|^2\]\bigg\}^\frac{1}{2}
                                          + |\d Y_\th^\l(0)|.
                \end{aligned}
                \end{equation*}
              Making use of  (\ref{2.1-0}), (\ref{3.2}), \autoref{le 3.1} and \autoref{pro 3.2} (or \autoref{re 3.3}) we have
                \begin{equation*}
                \begin{aligned}
                0\leq \lim_{\l\ra0} |J_1^\l|
                \leq \lim_{\l\ra0} \sup_{\th\in \Th}\dbE\[ \int_0^1 \pa_x\phi_\th^{\l\rho}(T)d\rho\cd\ \d X_\th^\l(T)
                                          + \int_0^1 \pa_y\g_\th^{\l\rho}(0)d\rho\cd\ \d Y_\th^\l(0)\]=0.
                \end{aligned}
                \end{equation*}
            Moreover, based on (\ref{3.2}), \autoref{le 3.1}, (\ref{3.6}), \autoref{re 3.3}  and the assumption that
               $\pa_y\g_\th$ is Lipschitz continuous with respect to $y$ uniformly in $\th$, we have
            \begin{equation*}
                \begin{aligned}
                |J_2^\l|&\leq\sup_{\th\in\Th}   \dbE\[ |\int_0^1 (\pa_x\phi_\th^{\l\rho}(T)-\pa_x\phi_\th(\bar{X}_\th(T)))d\rho|\cd\   |X_\th^1(T)|\] \\
                        &\qq   + \int_0^1 \sup_{\th\in\Th}|\pa_y\g_\th^{\l\rho}(0)-\pa_y\g_\th(\bar{Y}_\th(0))| d\rho\cd\ \sup_{\th\in\Th}|Y_\th^1(0)|\\
                      &\leq  C\l
                      \dbE\[ \sup_{t\in[0,T]} |X_\th^1(t)|^2+ \sup_{t\in[0,T]} |\d X_\th^\l(t)|^2 \] \\
                       &\q    + \int_0^1 \sup_{\th\in\Th}|\pa_y\g_\th^{\l\rho}(0)-\pa_y\g_\th(\bar{Y}_\th(0))| d\rho\cd\ \sup_{\th\in\Th}|Y_\th^1(0)|\ra 0, \q \text{as}\ \l\ra0.
                \end{aligned}
                \end{equation*}
            From the above estimate,  (\ref{3.13}) follows immediately.

             \ms

            \textbf {Step 2.}  Find a subsequence $\l_n\rightarrow0$ such that
            \begin{equation*}
            \begin{aligned}
              \limsup_{\l\ra0}\frac{J(v^{\l}(\cd))-J(\bar{v}(\cd))}{\l}=\lim_{n\ra\i}\frac{J(v^{\l_n}(\cd))-J(\bar{v}(\cd))}{\l_n}.
             \end{aligned}
             \end{equation*}
             For each  $n\in \dbN$, due to $\cQ^{v^{\l_n}}$ being nonempty,  there exists a probability measure $Q^{\l_n}\in \cQ^{v^{\l_n}}$
             such that
             \begin{equation*}
                \begin{aligned}
                & J(v^{\l_n}(\cd))=\int_\Th \dbE\[ \phi_\th(X^{\l_n}_\th(T))+\gamma_\th(Y_\th^{\l_n}(0))\] Q^{\l_n}(d\th),\\
                & J(\bar{v} (\cd))\geq \int_\Th \dbE\[ \phi_\th(\bar{X}_\th(T))+\gamma_\th(\bar{Y} _\th(0))\] Q^{\l_n}(d\th).
                \end{aligned}
               \end{equation*}
             To see this,  first notice
             \begin{equation}\label{3.16}
                \begin{aligned}
                \frac{J(v^{\l_n}(\cd))-J(\bar{v} (\cd)) }{\l_n}
                &\leq \int_\Th   \frac{\dbE[ \phi_\th(X^{\l_n}_\th(T))- \phi_\th(\bar{X}_\th(T))]
                                  +\gamma_\th(Y_\th^{\l_n}(0))-\gamma_\th(\bar{Y} _\th(0)) } {\l_n} Q^{\l_n}(d\th)\\
               &= \int_\Th\dbE\[ \pa_x\phi_\th(\bar{X}_\th(T))X^1_\th(T)+ \pa_y\g_\th(\bar{Y}_\th(0))Y^1_\th(0)\]Q^{\l_n}(d\th)
                +J_1^{\l_n}+J_2^{\l_n},
                \end{aligned}
               \end{equation}
              where $J_1^{\l_n}, J_2^{\l_n}$ are given in (\ref{3.14}). Similar to \textbf{Step 1}, we can show
              \begin{equation*}
              \lim_{n\ra\i} \(|J_1^{\l_n}|+|J_2^{\l_n}|\)=0.
              \end{equation*}
             Since $\cQ$ is weakly compact, there exists a probability measure $\h Q\in \cQ$ such that the sequence $(Q^{\l_n})_{n\in\dbN}$
             (choosing its subsequence if necessary) converges weakly to $\h Q$.
           In addition, according to \autoref{th 2.3}, \autoref{pro 2.3}, \autoref{pro 3.1-1}, \autoref{pro 3.3}, (\ref{3.2}), (\ref{3.6}) and
           \autoref{ass 11},
            we have that  the mapping $\th\mapsto\dbE[ \pa_x\phi_\th(\bar{X}_\th(T))X^1_\th(T)+ \pa_y\g_\th(\bar{Y}_\th(0))Y^1_\th(0)]$ is continuous and bounded. As $n\ra\i$ in (\ref{3.16}), it yields
           \begin{equation*}
                \begin{aligned}
                &\limsup_{\l\ra0}\frac{J(v^{\l}(\cd))-J(\bar{v}(\cd))}{\l}= \lim_{n\ra\i}\frac{J(v^{\l_n}(\cd))-J(\bar{v} (\cd)) }{\l_n}\\
                &\leq  \int_\Th\dbE\[ \pa_x\phi_\th(\bar{X}_\th(T))X^1_\th(T)+ \pa_y\g_\th(\bar{Y}_\th(0))Y^1_\th(0)\]\widehat{Q} (d\th).
                \end{aligned}
               \end{equation*}

\ms

           \textbf{Step 3.}  We prove $\h Q \in \cQ^{\bar{v}}$. For this, we need to show
           \begin{equation*}
                \begin{aligned}
                J(\bar{v}(\cd))=\int_\Th\dbE [ \phi_\th(\bar{X}_\th(T))+\gamma_\th(\bar{Y} _\th(0))]\widehat{Q} (d\th).
                \end{aligned}
               \end{equation*}
              In fact, since  $|\pa_x\phi_\th^{\l_n\rho}(T)|\leq L_1(1+|\bar{X}_\th(T)|+|X^1_\th(T)|+|\d X^\l_\th(T)|)$ and
              $|\pa_y\g_\th^{\l_n\rho}(0)|\leq C$, it follows from (\ref{2.1-0}), (\ref{3.2}), \autoref{le 3.1}, (\ref{3.6}) and
              \autoref{re 3.3} that, for all $\th\in \Th$,
              \begin{equation*}
                \begin{aligned}
                 & \Big|\dbE[ \phi_\th(X^{\l_n}_\th(T))-\phi_\th(\bar{X}_\th(T))]
                   +\gamma_\th(Y_\th^{\l_n}(0))-\gamma_\th(\bar{Y} _\th(0))  \Big|\\
                 &= \l_n \Big|\dbE\[\int_0^1 \pa_x\phi_\th^{\l_n\rho}(T) d\rho\cd (X_\th^1(T)+\d X_\th^{\l_n}(T))\]   \Big|
                     +\l_n \Big| \int_0^1 \pa_y\g_\th^{\l_n\rho}(0)d\rho\cd (Y_\th^1(0)+\d Y_\th^{\l_n}(0))  \Big|\\
                 &\leq \l_n \bigg\{\dbE\[1+ |\bar{X}_\th(T)|^2+|X^1_\th(T)|^2+|\d X^{\l_n}_\th(T)|^2 \]
                    +|Y_\th^1(0)+\d Y_\th^{\l_n}(0)|\bigg\}
                \ra 0, \q \text{as}\q \l_n\ra0.
                \end{aligned}
               \end{equation*}
             Above inequality leads to
             \begin{equation*}
                \begin{aligned}
                              &  \lim_{n\ra \i}|J(v^{\l_n}(\cd))-J(\bar{v}(\cd))|\\
                              & \leq\lim_{n\ra \i} \int_{\Th} \Big|\dbE [ \phi_\th(X^{\l_n}_\th(T))-\phi_\th(\bar{X}_\th(T))]\Big|
                                +\Big|\gamma_\th(Y_\th^{\l_n}(0))-\gamma_\th(\bar{Y} _\th(0))\Big|Q^{\l_n}(d\th) \\
                              &  \leq \lim_{n\ra \i}\sup_{\th\in \Th}\dbE\[ \Big|\phi_\th(X^{\l_n}_\th(T))-\phi_\th(\bar{X}_\th(T))\Big|\]
                                +\Big|\gamma_\th(Y_\th^{\l_n}(0))-\gamma_\th(\bar{Y} _\th(0))  \Big|=0
                \end{aligned}
               \end{equation*}
            and
               \begin{equation*}
                \begin{aligned}
               &\lim_{n\ra\i} \Big|\int_{\Th} \( \dbE [ \phi_\th(X^{\l_n}_\th(T)) +\gamma_\th(Y_\th^{\l_n}(0))]
               - \dbE [ \phi_\th(\bar{X}_\th(T))+\gamma_\th(\bar{Y} _\th(0))] \) Q^{\l_n}(d\th)\Big|\\
               &\leq\lim_{n\ra\i} \int_{\Th}  \Big|\dbE [ \phi_\th(X^{\l_n}_\th(T))-\phi_\th(\bar{X}_\th(T))]\Big|
                                +\Big|\gamma_\th(Y_\th^{\l_n}(0))-\gamma_\th(\bar{Y} _\th(0))\Big|  Q^{\l_n}(d\th)\\
               &\leq \lim_{n\ra \i}\sup_{\th\in \Th}\bigg\{\dbE\[ \Big|\phi_\th(X^{\l_n}_\th(T))-\phi_\th(\bar{X}_\th(T))\Big|\]
                                + \Big|\gamma_\th(Y_\th^{\l_n}(0))-\gamma_\th(\bar{Y} _\th(0))  \Big|\bigg\}=0.
                \end{aligned}
               \end{equation*}
               Consequently, it yields
                \begin{equation*}
                \begin{aligned}
               J(\bar{v}(\cd))&=\lim_{n\ra\i}J(v^{\l_n}(\cd))
                =\lim_{n\ra\i}\int_{\Th}\dbE [ \phi_\th(X^{\l_n}_\th(T)) +\gamma_\th(Y_\th^{\l_n}(0))] Q^{\l_n}(d\th)\\
               & =\lim_{n\ra\i}\int_{\Th}\dbE [ \phi_\th(\bar{X}_\th(T))+\gamma_\th(\bar{Y} _\th(0))]Q^{\l_n}(d\th)
                 =\int_{\Th}\dbE [ \phi_\th(\bar{X}_\th(T))+\gamma_\th(\bar{Y} _\th(0))]\h{Q} (d\th).
                \end{aligned}
               \end{equation*}
\end{proof}

                From  \autoref{le 3.4}, we have the main theorem in this section.
               \begin{theorem}\label{th 3.4}\rm
                 Under \autoref{ass 1} and \autoref{ass 11}, there exists a probability $\bar{Q}\in \cQ^{\bar{v}}$ such that, for
                 any $v(\cd)\in \cV_{0,T}$,
               \begin{equation*}
             \begin{aligned}
             &\int_\Th\dbE\[\pa_x\phi_\th(\bar{X}_\th(T))X^{1}_\th(T)+\pa_y\gamma_\th(\bar{Y}_\th(0)) Y^{1}_\th(0) \]\bar Q(d\th)\geq0.
             \end{aligned}
             \end{equation*}
               \end{theorem}
               \begin{proof}
               For a more precise expression, we write $X_\th^1, Y_\th^1$ as  $X_\th^{1,v}, Y_\th^{1,v}$.
               It follows \autoref{le 3.4} that,
               for any $v(\cd)\in \cV_{ad}$,
              \begin{equation*}
             \begin{aligned}
             0\leq \lim_{\l\ra0}\frac{1}{\l}\(J(v^\l(\cd))-J(\bar{v}(\cd)) \)
              =\sup_{Q\in \cQ^{\bar{v}}}\int_\Th\dbE\[\pa_x\phi_\th(\bar{X}_\th(T))X^{1,v}_\th(T) +\pa_y\gamma_\th(\bar{Y}_\th(0)) Y^{1,v}_\th(0) \]Q(d\th),
             \end{aligned}
             \end{equation*}
              which implies
               \begin{equation*}
             \begin{aligned}
             \inf_{v(\cd)\in\cV_{ad}}\sup_{Q\in \cQ^{\bar{v}}}\int_\Th\dbE\[\pa_x\phi_\th(\bar{X}_\th(T))X^{1,v}_\th(T) +\pa_y\gamma_\th(\bar{Y}_\th(0)) Y^{1,v}_\th(0) \]Q(d\th)\geq0.
             \end{aligned}
             \end{equation*}

             Define $\Upsilon_\th^v:=\dbE\[\pa_x\phi_\th(\bar{X}_\th(T))X^{1,v}_\th(T) +\pa_y\gamma_\th(\bar{Y}_\th(0)) Y^{1,v}_\th(0) \]$.
             From the linearity of $X^{1,v}, Y^{1,v}_\th$ with respect to $v$, we know that $v\mapsto \Upsilon_\th^v$ is a  linear mapping, i.e.,
             for $v(\cd),v'(\cd)\in\cV_{ad}, 0<l<1$ and $\th\in\Th$,
             \begin{equation*}
             \begin{aligned}
             \Upsilon_\th^{l v+(1-l)v'}=l\Upsilon_\th^{v}+(1-l)\Upsilon_\th^{v'}.
             \end{aligned}
             \end{equation*}
             In addition, for $v(\cd),v'(\cd)\in\cV_{ad}$ and $\th\in\Th,$
             \begin{equation*}
             \begin{aligned}
|\Upsilon_\th^{v}-\Upsilon_\th^{v'}| &\leq    C\dbE\[(1+|\bar{X}_\th(T)|)|X_\th^{1,v}(T)-X_\th^{1,v'}(T)|\]+|Y_\th^{1,v}(0)-Y_\th^{1,v'}(0)|\\
             &\leq C\Big\{\dbE[1+|\bar{X}_\th(T)|^2]\Big\}^\frac{1}{2}
                   \cd\Big\{\dbE[|X_\th^{1,v}(T)-X_\th^{1,v'}(T)|^2]\Big\}^\frac{1}{2}
             +|Y_\th^{1,v}(0)-Y_\th^{1,v'}(0)|.
             \end{aligned}
             \end{equation*}
             This, together with (\ref{3.2}) and (\ref{3.6}), can show that  $\Upsilon_\th^{v}$ is continuous with respect to  $v$ uniformly in $\th$.
             Making use of Sion's minimax theorem, we obtain
               \begin{equation*}
             \begin{aligned}
             &\inf_{v(\cd)\in\cV_{ad}}\sup_{Q\in \cQ^{\bar{v}}}\int_\Th\dbE\[\pa_x\phi_\th(\bar{X}_\th(T))X^{1,v}_\th(T) +\pa_y\gamma_\th(\bar{Y}_\th(0)) Y^{1,v}_\th(0) \]Q(d\th)\\
             &=\sup_{Q\in \cQ^{\bar{v}}}\inf_{v(\cd)\in\cV_{ad}}\int_\Th\dbE\[\pa_x\phi_\th(\bar{X}_\th(T))X^{1,v}_\th(T) +\pa_y\gamma_\th(\bar{Y}_\th(0)) Y^{1,v}_\th(0) \]Q(d\th)\geq0.
             \end{aligned}
             \end{equation*}
             Consequently, for any $\epsilon>0$, there exists a $Q^\epsilon\in \cQ^{\bar{v}}$ such that
              \begin{equation*}
             \begin{aligned}
             \inf_{v(\cd)\in\cV_{ad}}\int_\Th\dbE\[\pa_x\phi_\th(\bar{X}_\th(T))X^{1,v}_\th(T) +\pa_y\gamma_\th(\bar{Y}_\th(0)) Y^{1,v}_\th(0) \]Q^\epsilon(d\th)\geq-\epsilon.
             \end{aligned}
             \end{equation*}
             Finally, the compactness of $\cQ^{\bar{v}}$ allows us to take a subsequence $\epsilon_n\ra 0$ such that
             $Q^{\epsilon_n}$ converges weakly to a probability measure $\bar Q\in \cQ^{\bar{v}}$ and, for all $v(\cd)\in\cV_{ad}$,
                \begin{equation*}
             \begin{aligned}
             & \int_\Th\dbE\[\pa_x\phi_\th(\bar{X}_\th(T))X^{1,v}_\th(T) +\pa_y\gamma_\th(\bar{Y}_\th(0)) Y^{1,v}_\th(0) \]\bar Q(d\th)\\
             &=\lim_{\epsilon_n\ra0}\int_\Th\dbE\[\pa_x\phi_\th(\bar{X}_\th(T))X^{1,v}_\th(T) +\pa_y\gamma_\th(\bar{Y}_\th(0)) Y^{1,v}_\th(0) \]Q^{\epsilon_n}(d\th)\geq0.
             \end{aligned}
             \end{equation*}
               \end{proof}

 \section{Necessary and sufficient maximum principles}\label{sec 4}

        The necessary and sufficient maximum principles are discussed in this section.
        For this, we introduce the adjoint equation:
         \begin{equation}\label{5.1}
         \left\{
                \begin{aligned}
               dp_\th(t) & = \pa_yf_\th(t)p_\th(t)dt + \pa_zf_\th(t)p_\th(t)dW(t),\ \ \ t\in[0,T],\\
              -dq_\th(t) & =\(-\pa_xf_\th(t)p_\th(t)+(\pa_xb_\th(t))^\top q_\th(t)
              +\sum_{i=1}^d(\pa^i_{x}\si_\th(t))^\top   r^i_\th(t) -\dbE[\pa_{x'}f_\th(t)p_\th(t)] \\
                         &\q  +\dbE[(\pa_{x'}b_\th(t))^{\top} q_\th(t)]\)dt-r_\th(t)dW(t),\ \ \ t\in[0,T],\\
               p_\th(0)  & =-\pa_y\gamma_\th(\bar Y_\th(0)),\\
               q_\th(T)  &=-(\pa_x\F_\th(\bar X_\th(T),\dbE[\bar X_\th(T)] ))^\top p_\th(T)
                             -\dbE\[\pa_{x'}\F_\th(\bar X_\th(T),\dbE[\bar X_\th(T)] ))^\top p_\th(T)    \]\\
                         & \q +\pa_x\phi_\th(\bar X_\th(T)).
                \end{aligned}
                \right.
               \end{equation}
               The forward equation in (\ref{5.1}) is a $1$-dimensional SDE with unbounded coefficient.

               The solvability of (\ref{5.1}) comes from the following lemma which is an immediate adaptation of \cite[Lemma 3 and Lemma 4]{Galchuk-1978}.
                \begin{lemma}\label{th 6.1}\rm
                 Assume
             $ \dbE\[\int_0^T |\varphi_1(t)|^pdt\]<\i, \dbE\(\int_0^T |\varphi_2(t)|^2dt\)^\frac{p}{2}<\i,$  and  $a_1\cd W$, $a_2\cd  W\in  \text{BMO}$.
            Then for $p>1$, the $1$-dimensional SDE with  unbounded coefficient below
            \begin{equation}\label{6.1}
            \left\{
            \begin{aligned}
                  dX(t) & = [a_1(t)X(t)+\varphi_1(t)]dt + [a_2(t)X(t)+\varphi_2(t)]dW(t),\  t\in[0,T],\\
                    X(0)  & =x_0\in \dbR.
             \end{aligned}
             \right.
             \end{equation}
 has  a unique solution  $X(\cd)\in\cS_\dbF^p([0,T];\dbR)$.
               \end{lemma}

                \begin{proof}

              Define
              \begin{equation*}
              \begin{aligned}
              \t_0=0,\q \t_{n+1}&=\inf\bigg\{t>\t_n: \(\int_{\t_n}^t|a_2(s)|^2ds\)^\frac{p}{2}\geq
              2^{-p-1} C_p^{-1} \q  \text{or}\
                \int_{\t_n}^t|a_1(s)|^pds\geq 2^{-p-1}
              \bigg\}\wedge T,
              \end{aligned}
              \end{equation*}
              where $C_p$ only depending on $p$ is the constant in   Buckholder-Davis-Gundy inequality.
              Clearly, $\t_n\uparrow T$,  as $n\uparrow \i$.

              The proof is split into two steps.

                \textbf{ Step 1.} We show that the equation (\ref{6.1}) possesses a unique solution $X_1(\cd)\in \cS_\dbF^p([\t_0,\t_1];\dbR)$
                on the interval $[\t_0, \t_1]$ with the fixed point theorem.
                Let $X\in \cS_{\dbF}^p([\t_0, \t_1];\dbR)$. Define
                \begin{equation*}\label{6.2}
                \left\{
                \begin{aligned}
                 \Gamma(X)(t)&:=x_0+\int_{\t_0}^{t}   \[a_1(s)  X(s)+\varphi_1(s)\]ds+\int_{\t_0}^{t}  \[a_2(s)  X(s)+\varphi_2(s)\]dW(s),\  t\in[\t_0,\t_1],\\
                    \Gamma(X)(\t_0) &:=x_0.
                    \end{aligned}
                    \right.
                \end{equation*}
                It is easy to check $ \Gamma(X)\in \cS_{\dbF}^p([\t_0, \t_1];\dbR)$.
                For  $X',  X''\in \cS_{\dbF}^p([\t_0, \t_1];\dbR)$, define
                \begin{equation*}
                \D  X= X'-X'', \q   \G(\D X)=\G(X')-\G(X'').
                \end{equation*}
                Then it follows that
                \begin{equation*}
                \left\{
                \begin{aligned}
                 \G(\D X)(t)&=\int_{\t_0}^t a_1(s)\D X(s)ds
                               +\int_{\t_0}^t a_2(s) \D X(s)dW(s), \q
                       t\in[\t_0,\t_1],\\
                  \G(\D X)(\t_0)&=0.
                     \end{aligned}
                     \right.
                \end{equation*}
                By Buckholder-Davis-Gundy inequality and H\"{o}lder inequality, we have, for $p>1$
                \begin{equation*}
                \begin{aligned}
                \dbE\Big[\sup_{t\in[\t_0,\t_1]}|\G(\D X)(t)|^p\Big]
                  & \leq 2^{p-1} \dbE\Big[\int_{\t_0}^{\t_1}|a_1(s)|^pds\cd\sup_{t\in [\t_0,\t_1]}|\D X(t)|^p\Big]\\
                    &\ \ \ +2^{p-1}C_p\dbE\Big[\(\int_{\t_0}^{\t_1}|a_2(s)|^2ds\)^{\frac{p}{2}}\cd\sup_{t\in [\t_0,\t_1]}|\D X(t)|^p\Big].
                   \end{aligned}
                \end{equation*}
                Recalling the definition of $\t_1$, we have
                \begin{equation*}
                \dbE\Big[\sup_{t\in[\t_0,\t_1]}|\G(\D X)(t)|^p\Big]\leq \frac{1}{2}\dbE\Big[\sup_{t\in [\t_0,\t_1]}|\D X(t)|^p\Big].                \end{equation*}
               By the fixed point theorem, it yields that SDE (\ref{6.1}) has a unique solution $X_1(\cd)\in \cS_{\dbF}^p([\t_0, \t_1];\dbR)$ on the interval $[\t_0,\t_1].$

                \textbf{ Step 2.} Assuming that a solution $X$ of (\ref{6.1}) has been constructed on $[\t_0,\t_n]$, we
                 construct it on $(\t_n, \t_{n+1}]$.
                %
                 For this, define
                \begin{equation}\label{6.6}
                \begin{aligned}
                  \G(X)(t):=X(\t_n)+\int_{\t_n}^t\[a_1(s)  X(s) +\varphi_1(s)\]ds+\int_{\t_n}^t \[a_2(s)   X(s)+\varphi_2(s)\]dW(s),\  t\in[\t_n,\t_{n+1}].
                    \end{aligned}
                \end{equation}

                 According to the definition of $\t_{n+1}$, similar to \textbf{Step 1}, we know that (\ref{6.6}) has a unique solution
                   $X_{n+1}\in \cS_{\dbF}^p([\t_n, \t_{n+1}];\dbR)$. Repeating this procedure time and again, we get a solution of (\ref{6.1}) on
                   $[0,T]$. Moreover, this solution is unique by the procedure of construction.
 \end{proof}

               According to \autoref{th 6.1}, the forward equation in (\ref{5.1}) has a unique strong solution
               $p(\cd)\in\cS_\dbF^p([0,T];\dbR)$ for $p>1$.
               The backward equation   in (\ref{5.1}) is a $n$-dimensional mean-field BSDE with bounded Lipschitz
               coefficient. According to \cite[ Theorem 3.4]{Chen-Xing-Zhang-2020},
                for $p>1$, it has a unique solution
               $(q(\cd),r(\cd))\in \cS_\dbF^p([0,T];\dbR^n)\ts \cH_\dbF^{2,p}([0,T];\dbR^{n\ts d})$.

                  The following lemma show that $p_\th$ is continuous with respect to $\th$.

                \begin{lemma}\label{le 4.1}\rm
                Under \autoref{ass 1}, for $1<p<(p_{\pa_z f_\th}\wedge p_{(\pa_z f_{\tilde{\th}}-\pa_z f_{\th})})$,
                \begin{equation*}
                \begin{aligned}
                &    \lim_{\epsilon\ra 0}\sup_{d(\th, \tilde{\th})\leq\epsilon }\dbE\bigg[\sup_{t\in[0,T]}|p_\th(t)-p_{\tilde{\th}}(t)|^p\bigg]=0.
                 \end{aligned}
                \end{equation*}
                \end{lemma}

               \begin{proof}
                     Since $|p_\th(t)|^2>0$ for $t\in[0,T]$, applying It\^{o}'s formula to $\ln |p_\th(t)|^2$ on $[0,T]$,
                    we have
                     \begin{equation*}
                     p_\th(t)=\pm p_\th(0)\cd e^{\int_0^t\pa_yf_\th(s)ds}\cd \cE(\int_0^t\pa_zf_\th(s)^\top dW(s)).
                     \end{equation*}
                     We only show $p_\th(t)=-p_\th(0)\cd e^{\int_0^t\pa_yf_\th(s)ds}\cd \cE(\int_0^t\pa_zf_\th(s)^\top dW(s))$, i.e.,
                     \begin{equation*}
                     p_\th(t)= \pa_y\gamma_\th(\bar{Y}_\th(0))   \cd e^{\int_0^t\pa_yf_\th(s)ds}\cd \cE(\int_0^t\pa_zf_\th(s)^\top dW(s)).
                     \end{equation*}
                     Notice
                     \begin{equation*}
                      p_\th(t)-p_{\tilde{\th}}(t)=I_1(t)+I_2(t)+I_3(t),
                     \end{equation*}
                     where
                     \begin{equation*}
                     \begin{aligned}
                     I_1(t)&:=\[\pa_y\gamma_\th(\bar{Y}_\th(0))- \pa_y\gamma_{\tilde{\th}}(\bar{Y}_{\tilde{\th}}(0))\] \cd e^{\int_0^t\pa_yf_\th(s)ds}\cd \cE(\int_0^t\pa_zf_\th(s)^\top dW(s)),\\
                     I_2(t)&:=\pa_y\gamma_{\tilde{\th}}(\bar{Y}_{\tilde{\th}}(0))\cd
                     \[e^{\int_0^t\pa_yf_\th(s)ds}-e^{\int_0^t\pa_yf_{\tilde{\th}}(s)ds}\]\cd \cE(\int_0^t\pa_zf_\th(s)^\top dW(s)),\\
                      I_3(t)&:=\pa_y\gamma_{\tilde{\th}}(\bar{Y}_{\tilde{\th}}(0))\cd
                     e^{\int_0^t\pa_yf_{\tilde{\th}}(s)ds}\cd
                     \[\cE(\int_0^t\pa_zf_\th(s)^\top dW(s))-\cE(\int_0^t\pa_zf_{\tilde{\th}}(s)^\top dW(s))\] .
                     \end{aligned}
                     \end{equation*}

                     For $I_1(t)$, since
                     \begin{equation*}
                     \begin{aligned}
                     |\pa_y\gamma_\th(\bar{Y}_\th(0))- \pa_y\gamma_{\tilde{\th}}(\bar{Y}_{\tilde{\th}}(0))|
                     & \leq  |\pa_y\gamma_\th(\bar{Y}_\th(0))- \pa_y\gamma_{\tilde{\th}}(\bar{Y}_{\th}(0))|
                     +|\pa_y\gamma_{\tilde{\th}}(\bar{Y}_{\th}(0))-\pa_y\gamma_{\tilde{\th}}(\bar{Y}_{\tilde{\th}}(0))|\\
                     &\leq d(\th, \tilde{\th})+C |\bar{Y}_\th(0)-\bar{Y}_{\tilde{\th}}(0)|,
                     \end{aligned}
                     \end{equation*}
                     it yields from   Doob's inequality and reverse H\"older's inequality that
                     \begin{equation*}
                     \begin{aligned}
                      \dbE\[\sup\limits_{t\in[0,T]}|I_1(t)|^p\]
                     &\leq C_p\dbE\[\sup\limits_{t\in[0,T]}\Big\{ (d^p(\th, \tilde{\th})+C |\bar{Y}_\th(0)-\bar{Y}_{\tilde{\th}}(0)|^p)\cd e^{pTC_1}
                        \cd\cE(\int_0^t\pa_zf_\th(s)^\top dW(s))^p\Big\}  \]\\
                      &\leq C_p d^p(\th, \tilde{\th}) \dbE\[\sup\limits_{t\in[0,T]}  \cE(\int_0^t\pa_zf_\th(s)^\top dW(s))^p  \]\\
                      &\q  +C_p \Big\{\dbE\[\sup\limits_{t\in[0,T]}|\bar{Y}_\th(t)-\bar{Y}_{\tilde{\th}}(t)|^{pq'}  \]   \Big\}^{\frac{1}{q'}}
                      \cd \Big\{\dbE\[\sup\limits_{t\in[0,T]}\cE(\int_0^t\pa_zf_\th(s)^\top dW(s))^{pp'}  \]   \Big\}^{\frac{1}{p'}}\\
                       &\leq C_p d^p(\th, \tilde{\th}) \dbE\[ \cE(\int_0^T\pa_zf_\th(s)^\top dW(s))^p  \]\\
                      &\q  +C_p \Big\{\dbE[\sup\limits_{t\in[0,T]}|\bar{Y}_\th(t)-\bar{Y}_{\tilde{\th}}(t)|^{pq'}  ]   \Big\}^{\frac{1}{q'}}
                      \cd \Big\{\dbE[\cE(\int_0^T\pa_zf_\th(s)^\top dW(s))^{pp'}  ]   \Big\}^{\frac{1}{p'}}\\
                      &\leq C_p   d^p(\th, \tilde{\th})\cd K(p, \| \pa_zf_\th\cd W  \|_{BMO})\\
                      &\q +
                      \Big\{\dbE[\sup\limits_{t\in[0,T]}|\bar{Y}_\th(t)-\bar{Y}_{\tilde{\th}}(t)|^{pq'}  ]   \Big\}^{\frac{1}{q'}}
                      \cd  K(pp', \| \pa_zf_\th\cd W  \|_{BMO}),
                     \end{aligned}
                     \end{equation*}
                   where $p'=\frac{p+p_{\pa_zf_\th}}{2p} $ and $q'=\frac{p'}{p'-1}$. Clearly $pq'>p^*_{\pa_zf_\th}$, and by \autoref{pro 2.3}
                   we have, for  $1<p<p_{\pa_zf_\th}$,
                     \begin{equation*}
                      \lim_{\epsilon\ra 0}\sup_{d(\th, \tilde{\th})\leq\epsilon } \dbE\[\sup\limits_{t\in[0,T]}|I_1(t)|^p\]=0.
                      \end{equation*}

                       For $I_2(t)$, since
  \begin{equation*}
                     \[e^{\int_0^t\pa_yf_\th(s)ds}-e^{\int_0^t\pa_yf_{\tilde{\th}}(s)ds}\] \leq e^{C_1T}\cd\[e^{Td(\th,\tilde{\th})}-1\],
                     \end{equation*}
we know, for $1<p<p_{\pa_zf_\th}$,
  \begin{equation*}
  \begin{aligned}
                    \dbE\[\sup\limits_{t\in[0,T]}|I_2(t)|^p\]  &\leq Le^{C_1 T}\cd\[e^{Td(\th,\tilde{\th})}-1\]\cd \dbE\[\sup\limits_{t\in[0,T]}  \cE(\int_0^t\pa_zf_\th(s)^\top dW(s))^p  \]\\
                    & \leq Le^{C_1T}\cd\[e^{Td(\th,\tilde{\th})}-1\] \cd   K(p, \| \pa_zf_\th\cd W  \|_{BMO}).
                    \end{aligned}
                     \end{equation*}

   Finally, we prove, for
$1<p<(p_{\pa_z f_\th}\wedge p_{(\pa_z f_{\tilde{\th}}-\pa_z f_{\th})})$,
                     \begin{equation*}
                      \lim_{\epsilon\ra 0}\sup_{d(\th, \tilde{\th})\leq\epsilon } \dbE\[\sup\limits_{t\in[0,T]}|I_3(t)|^p\]=0.
                      \end{equation*}
                      In fact, due to $|\pa_y\gamma_{\tilde{\th}}(\bar{Y}_{\tilde{\th}}(0))\cd
                     e^{\int_0^t\pa_yf_{\tilde{\th}}(s)ds}|\leq C $, we only need to prove, for
    $1<p<(p_{\pa_z f_\th}\wedge p_{(\pa_z f_{\tilde{\th}}-\pa_z f_{\th})}) $,
                     \begin{equation*}
                      \lim_{\epsilon\ra 0}\sup_{d(\th, \tilde{\th})\leq\epsilon } \dbE\[\sup\limits_{t\in[0,T]}|\cE(\int_0^t\pa_zf_\th(s)^\top dW(s))-\cE(\int_0^t\pa_zf_{\tilde{\th}}(s)^\top dW(s))|^p\]=0.
                      \end{equation*}
    For convenience, set $X_\th(t)=\cE(\int_0^t\pa_zf_\th(s)^\top dW(s))$, and $X_\th$ satisfies the following SDE
    \begin{equation*}
    X_\th(t)=1+\int_0^t\pa_zf_\th(s)X_\th(s)dW(s),\ t\in[0,T].
    \end{equation*}
    Set $\bar{X}(s)=X_\th(s)-X_{\tilde{\th}}(s)$  and $\varphi(s)=(\pa_zf_\th(s)-\pa_zf_{\tilde{\th}}(s))X_{\tilde{\th}}(s)$. Then
     \begin{equation*}
      \bar{X}(t)=\int_0^t\(\pa_zf_\th(s)\bar{X}(s)+\varphi(s)\)dW(s), \q t\in[0,T].
      \end{equation*}
      Consider
      \begin{equation*}
      \left\{
      \begin{aligned}
      d\Pi_\th(t)&=(\pa_zf_\th(t))^2\Pi_\th(t)dt- \pa_zf_\th(t)\Pi_\th(t)dW(t), \q t\in[0,T],\\
      \Pi_\th(0)&=1.
      \end{aligned}
      \right.
      \end{equation*}
   By \autoref{th 6.1}, the above equation has a unique strong solution $\Pi_\th\in \cS^p([0,T];\dbR)$ for $p>1$.
     Then it follows from It\^{o}'s formula to $\bar{X}(s)\Pi_\th(s)$ on $[0,t]$  that
       \begin{equation}\label{4.2}
       \bar{X}(t)\Pi_\th(t)=\int_0^t(-\pa_zf_\th(s)\Pi_\th(s)\f(s))ds+\int_0^t\Pi_\th(s)\f(s)dW(s), \q t\in[0,T].
       \end{equation}
        In addition, consider
          \begin{equation*}
      \left\{
      \begin{aligned}
      d\pi_\th(t)&=\pa_zf_\th(t)\pi_\th(t)dW(t), \q t\in[0,T],\\
      \pi_\th(0)&=1.
      \end{aligned}
      \right.
      \end{equation*}
     By \autoref{th 6.1} again, the above equation also has a unique strong solution $\pi_\th\in \cS^p([0,T];\dbR)$ for $p>1$.
       Moreover, we know from It\^{o}'s  formula that $\Pi_\th(t)\pi_\th(t)=1,\ t\in[0,T]$. Hence,
       \begin{equation*}
       \Pi_\th(t)^{-1}=\pi_\th(t)=\cE\(\int_0^t\pa_zf_\th(r)^\top dW(r)\).
       \end{equation*}
       From H\"{o}lder inequality we can see, for
    $1<p<(p_{\pa_z f_\th}\wedge p_{(\pa_z f_{\tilde{\th}}-\pa_z f_{\th})})$,
   \begin{equation*}
   \begin{aligned}
    &\dbE\[ \sup\limits_{t\in[0,T]}|\bar{X}(t)|^p \]=\dbE\[ \sup\limits_{t\in[0,T]}|\bar{X}(t)\Pi_\th(t)\pi_\th(t)|^p \]\\
    &\leq\dbE\[\sup\limits_{t\in[0,T]}|\bar{X}(t)\Pi_\th(t)|^p\cd \sup\limits_{t\in[0,T]}|\pi_\th(t)|^p \]\\
    &\leq \Big\{\dbE\[\sup\limits_{t\in[0,T]}|\bar{X}(t)\Pi_\th(t)|^{pq'} \]\Big\}^{\frac{1}{q'}}
     \Big\{\dbE\[\sup\limits_{t\in[0,T]}|\pi_\th(t)|^{pp'} \]\Big\}^{\frac{1}{p'}},
    \end{aligned}
   \end{equation*}
   where $p'=\frac{p+(p_{\pa_z f_\th}\wedge p_{(\pa_z f_{\tilde{\th}}-\pa_z f_{\th})} )}{2p}$
   and  $q'=\frac{p'}{p'-1}$.
  Since $pp'<p_{\pa_z f_\th}$,
  it follows from Doob's inequality and reverse H\"{o}lder inequality that
  \begin{equation*}
   \dbE\[ \sup\limits_{t\in[0,T]}\pi_\th(t)^{pp'}\]\leq C \dbE\[  \pi_\th(T)^{pp'}\]\leq C K(pp', \| \pa_zf_\th\cd W  \|_{BMO}).
   \end{equation*}
Consequently, by H\"{o}lder inequality we have
   \begin{equation}\label{4.43}
   \begin{aligned}
\dbE\[ \sup\limits_{t\in[0,T]}|\bar{X}(t)|^p \]
     &\leq C \Big\{\dbE\[\sup\limits_{t\in[0,T]}|\bar{X}(t)\Pi_\th(t)|^{pq'} \]\Big\}^{\frac{1}{q'}}\\
    &\leq C\Big\{\dbE\[\sup\limits_{t\in[0,T]}|\bar{X}(t)\Pi_\th(t)|^{p\frac{p'}{p'-1}} \]\Big\}^{\frac{1}{q'}}\\
    &\leq C\Big\{\dbE\[\sup\limits_{t\in[0,T]}|\bar{X}(t)\Pi_\th(t)|^{pp'} \]\Big\}^{\frac{1}{p'}},
    \end{aligned}
   \end{equation}
   where $C$ depends on $pp'$ and $\| \pa_zf_\th\cd W  \|_{BMO}$.
  Besides, from (\ref{4.2}) we arrive at
   \begin{equation}\label{4.44}
   \begin{aligned}
   \dbE\[\sup\limits_{t\in[0,T]}|\bar{X}(t)\Pi_\th(t)|^{pp'}\]&\leq C\dbE\[\int_0^T|-\pa_zf_\th(s)\Pi_\th(s)\f(s))|^{pp'}ds\]\\
                 &\q +C \dbE\[\sup\limits_{t\in[0,T]}|\int_0^t\Pi_\th(s)\f(s)dW(s)|^{pp'}\],
   \end{aligned}
    \end{equation}
where the constant $C>0$ depends on $pp'$.

Then we estimate the $ds$-term and the $dW$-term in the above inequality one by one. As for the $ds$-term,  since
    \begin{equation}\label{4.40}
    \left\{
   \begin{aligned}
    |\pa_zf_\th(s)|&\leq C(1+|\bar{Z}_\th(s)|),\\
    |\f(s)|&\leq d(\th, \tilde{\th})\cd \cE\( \int_0^s\pa_zf_{\tilde{\th}}(r)^\top dW(r)\),\\
    \Pi_\th(s)&=\exp\Big\{ \int_0^s(-\pa_zf_\th(r))^\top dW(r)+\frac{1}{2}\int_0^s|\pa_zf_\th(r)|^2dr \Big \},
   \end{aligned}
   \right.
    \end{equation}
by H\"{o}lder inequality we have
    \begin{equation}\label{4.45}
   \begin{aligned}
    &\dbE\[\int_0^T|-\pa_zf_\th(s)\Pi_\th(s)\f(s))|^{pp'}ds\] \\
    &\leq Cd(\th, \tilde{\th})^{pp'} \dbE\bigg[\(\int_0^T
    (1+|\bar{Z}_\th(s)|)\cd \cE\( \int_0^s( \pa_zf_{\tilde{\th}}(r)-\pa_zf_{\th}(r))^\top dW(r)\)\cd\\
    &\qq\qq\qq\qq\qq\qq
    e^{ \int_0^s\pa_zf_\th(r)^\top (\pa_zf_\th(r)-\pa_zf_{\tilde{\th}}(r))dr}ds\)^{pp'}\bigg]\\
    &\leq Cd(\th, \tilde{\th})^{pp'} \dbE\[
    \sup\limits_{s\in[0,T]}\cE\( \int_0^s( \pa_zf_{\tilde{\th}}(r)-\pa_zf_{\th}(r))^\top dW(r)\)^{pp'}\cd \\
    &\qq\qq\qq\qq \qq\qq\qq\qq
    e^{C d(\th,\tilde{\th})pp'\int_0^T(1+|\bar{Z}_\th(r)|)dr}\cd
    \(\int_0^T   (1+|\bar{Z}_\th(r)|)dr\)^{pp'}\]\\
    &\leq Cd(\th, \tilde{\th})^{pp'} \bigg\{\dbE\[
    \sup\limits_{s\in[0,T]}  \cE\( \int_0^s( \pa_zf_{\tilde{\th}}(r)-\pa_zf_{\th}(r))^\top dW(r)\)^{pp'p''}\]  \bigg\}^{\frac{1}{p''}}\cd\\
    &\qq\qq\qq
     \bigg\{\dbE\[e^{C d(\th,\tilde{\th})pp'q''\int_0^T(1+|\bar{Z}_\th(r)|)dr}\cd
    \(\int_0^T   (1+|\bar{Z}_\th(r)|)dr\)^{pp'q''}\]\bigg\}^{\frac{1}{q''}},
   \end{aligned}
    \end{equation}
    where
    $p''=\frac{p'+(p_{\pa_z f_\th}\wedge p_{(\pa_z f_{\tilde{\th}}-\pa_z f_{\th})})}{2pp'}$
    and $q''=\frac{p''}{p''-1}$.
   Since  $|\pa_z f_{\tilde{\th}}(s)-\pa_z f_{\th}(s)|\leq C(1+|\bar{Z}_\th(s)|+|\bar{Z}_{\tilde{\th}}(s)|)$ we know
   $(\pa_z f_{\tilde{\th}} -\pa_z f_{\th})\cd W\in BMO$.
   Note that  $1<p'<pp'p''<p_{(\pa_z f_{\tilde{\th}}-\pa_z f_{\th})}$.  It follows from Doob's inequality and
   reverse H\"{o}lder inequality that
   \begin{equation*}
   \dbE\[
    \sup\limits_{s\in[0,T]}  \cE\( \int_0^s( \pa_zf_{\tilde{\th}}(r)-\pa_zf_{\th}(r))^\top  dW(r)\)^{pp'p''}\]
    \leq  K(pp'p'', \| (\pa_zf_{\tilde{\th}}-\pa_zf_{\th})\cd W  \|_{BMO}).
   \end{equation*}
     On the other hand, by H\"{o}lder inequality and $\bar{Z}_\th\cd W\in BMO$, we have
         \begin{equation*}
   \begin{aligned}
      & \bigg\{\dbE\[e^{C d(\th,\tilde{\th})pp'q''\int_0^T(1+|\bar{Z}_\th(r)|)dr}\cd
              \(\int_0^T   (1+|\bar{Z}_\th(r)|)dr\)^{pp'q''}\]\bigg\}^{\frac{1}{q''}}\\
      &\leq C \bigg\{\dbE\[e^{2C d(\th,\tilde{\th})pp'q''\int_0^T(1+|\bar{Z}_\th(r)|)dr}\] \bigg\}^{\frac{1}{2q''}}\cd
      \bigg\{\dbE\[
              \(\int_0^T   (1+|\bar{Z}_\th(r)|^2)dr\)^{pp'q''}\]\bigg\}^{\frac{1}{2q''}}\\
       &\leq C  \bigg\{\dbE\[e^{2C d(\th,\tilde{\th})pp'q''\int_0^T(1+|\bar{Z}_\th(r)|)dr} \]\bigg\}^{\frac{1}{2q''}}.
   \end{aligned}
    \end{equation*}
  Without loss of generality, assume $d(\th, \tilde{\th} )< 1$.  Consequently, it follows from  John-Nirenberg inquality
    that

         \begin{equation*}
         \begin{aligned}
         &\dbE\[e^{2C d(\th,\tilde{\th})pp'q''\int_0^T(1+|\bar{Z}_\th(r)|)dr} \] < \dbE\[e^{2C pp'q''\int_0^T(1+|\bar{Z}_\th(r)|)dr} \] \\
         &\leq e^{2C pp'q''(T+\frac{T}{\delta})}
         \dbE\[e^{2C pp'q''\frac{\d}{4}\int_0^T  |\bar{Z}_\th(r)|^2dr} \]\\
         &\leq e^{2C pp'q''(T+\frac{T}{\delta})} \(1- 2C pp'q''\frac{\d}{4} \|\bar{Z}_\th\cd W \|_{BMO}^{2}   \) ^{-1}\\
         &=\frac{4}{3}e^{2C pp'q''(T+\frac{T}{\delta})},
         \end{aligned}
         \end{equation*}
         where $\d=\frac{\|\bar{Z}_\th\cd W \|_{BMO}^{-2}}{2Cpp'q''}$.
        Combining the above inequalities, we have
         \begin{equation}\label{4.47}
         \dbE\[\int_0^T|-\pa_zf_\th(s)\Pi_\th(s)\f(s))|^{pp'}ds\]\leq Cd(\th, \tilde{\th})^{pp'}.
           \end{equation}

         Next, we analyse the $dW$-term. Similar to the  $ds$-term,  based on  Buckholder-Davis-Gundy inequality, (\ref{4.40}), H\"{o}lder inequality and John-Nirenberg inequality, we have
         \begin{equation}\label{4.48}
         \begin{aligned}
         &\dbE\[\sup\limits_{t\in[0,T]}|\int_0^t\Pi_\th(s)\f(s)dW(s)|^{pp'}\]\\
         &\leq  Cd(\th,\tilde{\th})^{pp'} \dbE\[
         \sup\limits_{s\in[0,T]} \cE\( \int_0^s( \pa_zf_{\tilde{\th}}(r)-\pa_zf_{\th}(r))^\top dW(r)\)^{pp'}\cd
         e^{Cpp'\int_0^T(1+|\bar{Z}_\th(r)|)d(\th,\tilde{\th})dr}\]\\
         &\leq  Cd(\th,\tilde{\th})^{pp'} \bigg\{\dbE\[
         \sup\limits_{s\in[0,T]} \cE\( \int_0^s( \pa_zf_{\tilde{\th}}(r)-\pa_zf_{\th}(r))^\top dW(r)\)^{pp'p''}\]\bigg\}^\frac{1}{p''}
         \cd\\
         &\qq\qq\qq\qq\qq\qq\qq\qq\qq\qq
         \bigg\{\dbE\[e^{Cpp'q''\int_0^T(1+|\bar{Z}_\th(r)|)d(\th,\tilde{\th})dr}\]\bigg\}^\frac{1}{q''} \\
         &\leq Cd(\th,\tilde{\th})^{pp'},
         \end{aligned}
         \end{equation}
         where $C$ depends on $\| (\pa_zf_{\tilde{\th}}-\pa_zf_{\th})\cd W  \|_{BMO}\|$ and  $\pa_zf_{\th}\cd W  \|_{BMO}$, and  $p'',q''$ are given in (\ref{4.45}).\\
         Finally, taking into account (\ref{4.43}), (\ref{4.44}), (\ref{4.47}) and (\ref{4.48}), we have
    $\dbE\[ \sup\limits_{t\in[0,T]}|\bar{X}(t)|^p \]\leq Cd(\th,\tilde{\th})^p.$
    The proof is complete.
\end{proof}

Before we prove the stochastic maximum principle, we need to give the continuity  of $q_\th,r_\th$ with respect to $\th$.
\begin{lemma}\label{le 4.2}\rm
                Under \autoref{ass 1}, we have,  for $1<p<(p_{\pa_z f_\th}\wedge p_{\pa_z f_{\tilde{\th}}-\pa_z f_{\th}})$,
                \begin{equation*}
                \begin{aligned}
                &    \lim_{\epsilon\ra 0}\sup_{d(\th, \tilde{\th})\leq\epsilon }\dbE\bigg[
                 \sup_{t\in[0,T]}|q_\th(t)-q_{\tilde{\th}}(t)|^p+\(\int_0^{T}|r_\th(t)-r_{\tilde{\th}}(t)|^2dt\)^\frac{p}{2}
                \bigg]=0.
                 \end{aligned}
                \end{equation*}
                \end{lemma}

\begin{proof}
                   Notice that
                   \begin{equation*}
                   \begin{aligned}
                   q_\th(t)-q_{\tilde{\th}}(t)&=[q_\th(T)-q_{\tilde{\th}}(T)]+
                   \int_t^T \(\pa_xb_\th(s)^\top(q_\th(s)-q_{\tilde{\th}}(s))
                   +\sum_{i=1}^d\pa_x\si^i_\th(s)^\top(r^i_\th(s)-r^i_{\tilde{\th}}(s))\\
                   &\q
                   +\dbE\[ \pa_{x'}b_\th(s)^\top(q_\th(s)-q_{\tilde{\th}}(s))\]+I_{1,\th,\tilde{\th}}(s)
                   +I_{2,\th,\tilde{\th}}(s)\)ds-\int_t^T(r_\th(s)-r_{\tilde{\th}}(s))
                   dW(s),
                   \end{aligned}
                   \end{equation*}
                   where
                   \begin{equation*}
                   \begin{aligned}
                   I_{1,\th,\tilde{\th}}(s)&=\pa_xf_{\tilde{\th}}(s)p_{\tilde{\th}}(s)-\pa_xf_\th(s)p_\th(s)
                                +\dbE[\pa_{x'}f_{\tilde{\th}}(s)p_{\tilde{\th}}(s)-\pa_{x'}f_\th(s)p_\th(s)],\\
                   I_{2,\th,\tilde{\th}}(s)&= (\pa_xb_\th(t)-\pa_xb_{\tilde{\th}}(t))^\top q_{\tilde{\th}}(t)
                                  +\sum_{i=1}^d(\pa_x\si^i_\th(t)-\pa_x\si^i_{\tilde{\th}}(t))^\top r^i_{\tilde{\th}}(t)+\dbE[(\pa_{x'}b_\th(t)-\pa_{x'}b_{\tilde{\th}}(t))^\top q_{\tilde{\th}}(t)].
                   \end{aligned}
                   \end{equation*}
                   Since $\pa_xb_\th, \pa_x\si^i_\th,\pa_{x'}b_\th$ are bounded,
                   the backward equation in (\ref{5.1}) is a mean-field BSDE with Lipschitz coefficient.  By
                   \cite[Theorem 3.3]{Chen-Xing-Zhang-2020}, it yields that for $p>1$,
                   \begin{equation*}
                   \begin{aligned}
                   &\dbE\bigg[\sup_{t\in[0,T]}|q_\th(t)-q_{\tilde{\th}}(t)|^p
                           +\(\int_0^{T}|r_\th(t)-r_{\tilde{\th}}(t)|^2dt\)^\frac{p}{2}\bigg]\\
                   &\leq C\dbE\bigg[ |q_\th(T)-q_{\tilde{\th}}(T)|^p
                         +\(\int_0^T|I_{1,\th,\tilde{\th}}(t)+I_{2,\th,\tilde{\th}}(t)|dt \)^p
                   \bigg].
                   \end{aligned}
                   \end{equation*}
                   Since $\pa_x\F_\th,\pa_{x'}\F_\th,\pa_xb_\th, \pa_x\si_\th,  \pa_{x'}b_\th,\pa_xf_\th, \pa_{x'}f_\th,\pa_x\phi_\th$ are bounded and Lipschitz continuous in
                   $\th$, it follows from \autoref{le 4.1} and \autoref{pro 2.3} that
                \begin{equation*}\label{4.4}
                \begin{aligned}
                &    \lim_{\epsilon\ra 0}\sup_{d(\th, \tilde{\th})\leq\epsilon }\dbE\bigg[
                \sup_{t\in[0,T]}|q_\th(t)-q_{\tilde{\th}}(t)|^p
                +\(\int_0^{T}|r_\th(t)-r_{\tilde{\th}}(t)|^2dt\)^\frac{p}{2}
                \bigg]=0.
                 \end{aligned}
                \end{equation*}
 \end{proof}

     Define the Hamiltonian: for $t\in[0,T],x,x'\in\dbR^n, y\in \dbR, z\in \dbR^d, p\in\dbR,q\in\dbR^n, r\in \dbR^{n\ts d}, v\in V,$
                 \begin{equation*}
                 \begin{aligned}
                 H_\th(t,x,x',y,z,v,p,q,r)&:=q^\top b_\th(t,x,x',v)+\sum_{i=1}^d(r^i)^\top \si_\th^i(t,x,v)-pf_\th(t,x,x',y,z,v)
                 \end{aligned}
                 \end{equation*}
and its partial derivative with respect to $v$:
                 \begin{equation}\label{5.4}
                  \begin{aligned}
                 \L_\th(t,\o):=\pa_vH_\th(t,\bar{X}_\th(t),   \dbE[\bar{X}_\th(t)],   \bar{Y}_\th(t),
                      \bar{Z}_\th(t),\bar{v}(t),
                    p_\th(t), q_\th(t), r_\th(t)).
                   \end{aligned}
                 \end{equation}

%
%

\begin{lemma}\label{le 7.1}\rm
             Under \autoref{ass 1} and \autoref{ass 11},  
            the above $\L$ is an $\cF$-progressively measurable
             process, i.e., for each $t\in[0,T]$, the function $\L_\th(t,\o): \Th\ts[0,t]\ts\O\ra\dbR$  is $\sB(\Th)\ts
             \sB([0,t])\ts\cF_t$-measurable.
             \end{lemma}

             \begin{proof}  Since $\Th$ is a Polish space, for each $M>1$, there exists  a compact subset $C^M\subset \Th$ such that
             $\bar Q(\th\notin C^M)\leq \frac{1}{M}$. Then  we can find a subsequence of open neighborhoods $\(B(\th_l, \frac{1}{2M})\)_{l=1}^{L_M}$ such that $C^M\subset \bigcup_{l=1}^{L_M}\(B(\th_l, \frac{1}{2M})\)$.
             By the  locally compact property of $\Th$ and using partitions of unity, there exists a sequence of continuous functions
             $h_l:\Th\ra \dbR$ with values in $[0,1]$ such that
             \begin{equation*}
              h_l(\th)=0,\ \text{if}\ \th\notin B\(\th_l, \frac{1}{2M}\),\  l=1,
              \cds, L_M,\ \text{and}\ \sum_{l=1}^{L_M}h_l(\th)=1,\ \text{if}\ \th\in C^M.
             \end{equation*}
             We choose some $\th_l^*$ satisfying  $h_l(\th_l^*)>0$  and define
             \begin{equation*}
               \L^M_\th(t):=\sum_{l=1}^{L_M}\L_{\th_l^*}(t)h_l(\th)\mathbf{1}_{\{\th\in C^M\}}.
             \end{equation*}
             Notice that
               $  \dbE\[ \int_0^T |\L_\th(t)|dt\]\leq L.$  It yields that
             \begin{equation*}
             \begin{aligned}
             &\int_\Th\dbE \[\int_0^T\Big|\L_\th^M(t)-\L_\th(t)\Big|dt\]\bar Q(d\th)\\
             &\leq \int_\Th \sum_{l=1}^{L_M}\dbE\bigg[\int_0^T |\L_\th^M(t)-\L_\th(t)|dt\bigg]h_l(\th) \mathbf{1}_{\{\th\in C^M\}}
                  +\dbE\bigg[\int_0^T|\L_\th(t)|dt\bigg]h_l(\th) \mathbf{1}_{\{\th\notin C^M\}} \bar Q(d\th)\\
             & \leq  \sup_{d(\th,\tilde{\th})\leq \frac{1}{M}}\dbE\bigg[\int_0^T |\L_{\tilde{\th}}(t)-\L_\th(t)|dt\bigg]
             +L\bar Q(\th\notin C^M) \\
             & \leq  \sup_{d(\th,\tilde{\th})\leq \frac{1}{M}}\dbE\bigg[\int_0^T |\L_{\tilde{\th}}(t)-\L_\th(t)|dt\bigg] +\frac{L}{M}.
             \end{aligned}
             \end{equation*}
             Here we use the fact  that  $h_l(\th)=0$ whenever $d(\th,\tilde{\th})\geq \frac{1}{2N}$.

             \ms

             Then we prove
            $
              \lim\limits_{M\ra\i}  \sup\limits_{d(\th,\tilde{\th})\leq \frac{1}{M}}\dbE[
              \int_0^T |\L_{\tilde{\th}}(t)-\L_\th(t)|dt]=0.
              $
             Recall that
             \begin{equation*}
             \begin{aligned}
             \L_\th(t,\o) & =\pa_vb_\th(t,\bar {X}_\th(t),\dbE[\bar {X}_\th(t)], \bar v(t))q_\th(t)
             +\sum_{i=1}^d\pa_v\si^i_\th(t,\bar {X}_\th(t),  \bar v(t))r^i_\th(t)\\
                           &\q -\pa_vf_\th(t,\bar {X}_\th(t),\dbE[\bar {X}_\th(t)],\bar {Y}_\th(t), \bar {Z}_\th(t), \bar {v}(t) )p_\th(t).
             \end{aligned}
             \end{equation*}
              Since $\pa_vb_\th, \pa_v\si^i_\th,\pa_vf_\th$ are bounded and Lipschitz continuous in $\th$, the desired result comes from \autoref{le 4.1} and \autoref{le 4.2}.

              \end{proof}

          Now we are ready to get the necessary condition of the optimal control.
                  \begin{theorem}\label{th 5.1}\rm (Stochastic Maximum Principle)
                  Let $\bar{v}(\cd)$ be an optimal control to the problem (\ref{2.10}), 
                 $(\bar{X}_\th(\cd),\bar{Y}_\th(\cd), \bar{Z}_\th(\cd) )$ be  the solution to the state equation (\ref{1.1}) with $\bar v(\cd)$,
                   and $(p_\th(\cd), q_\th(\cd), r_\th(\cd))$ be the solution to the adjoint equation (\ref{5.1}) with $\bar v(\cd)$. Under \autoref{ass 1} and \autoref{ass 11}, there exists a
                  probability measure $\bar{Q}\in\cQ^{\bar{v}}$ such that, for $v\in V$, $dt\ts d\dbP$-a.e.,
                  \begin{equation*}
                  \begin{aligned}
                  &  \int_\Th\langle\pa_vH_\th(t,\bar{X}_\th(t), \dbE[\bar{X}_\th(t)], \bar{Y}_\th(t),
                      \bar{Z}_\th(t), \bar{v}(t), p_\th(t), q_\th(t), r_\th(t)), v-\bar{v}(t)\rangle\bar{Q}(d\th)\geq0.
                  \end{aligned}
                  \end{equation*}
                  \end{theorem}
                  \begin{proof} Applying It\^{o}'s formula to   $(X^1_\th(\cd))^\top q_\th(\cd)+ Y^1_\th(\cd) p_\th(\cd)$, we have
                  \begin{equation*}
                  \begin{aligned}
                   &  \dbE\[   (\pa_x\phi_\th(\bar X_\th(T)))^\top X_\th^1(T)
                        +   \pa_y\g_\th (\bar Y_\th(0))Y_\th^1(0)   \] \\
                   & = \dbE\[\int_0^T  \langle (\pa_vb_\th(t))^\top q_\th(t) +\sum_{i=1}^d(\pa_v\si^i_\th(t))^\top r^i_\th(t)
                   -\pa_vf_\th(t)p_\th(t),    v(t)-\bar v(t) \rangle  dt\] \\
                   & = \dbE\[\int_0^T  \langle\pa_vH_\th(t,\bar{X}_\th(t), \dbE[\bar{X}_\th(t)],    \bar{Y}_\th(t),
                      \bar{Z}_\th(t),\bar{v}(t),
                    p_\th(t), q_\th(t), r_\th(t)),   v(t)-\bar v(t)\rangle dt\].
                   \end{aligned}
                  \end{equation*}
                 Thanks to the above equality and \autoref{th 3.4}, there exists some probability $ \bar Q\in \cQ^{\bar v}$ such that,  for all $v(\cd)\in \cV_{ad}$,
                 \begin{equation*}
                 \begin{aligned}
                 & \int_\Th\dbE\[\int_0^T  \langle\pa_vH_\th(t,\bar{X}_\th(t),  \dbE[\bar{X}_\th(t)],    \bar{Y}_\th(t),
                      \bar{Z}_\th(t),\bar{v}(t), p_\th(t), q_\th(t), r_\th(t)),   v(t)-\bar v(t)\rangle dt\]\bar Q(d\th)\geq0.
                  \end{aligned}
                 \end{equation*}
                Thanks to \autoref{le 7.1},
                  the process $ \L_\th(\cd)$ is $\cF$-progressively measurable.
                  Then it follows from  Fubini's Theorem that
                   \begin{equation*}
                 \begin{aligned}
                 &\dbE\[\int_0^T \int_\Th \langle\pa_vH_\th(t,\bar{X}_\th(t),  \dbE[\bar{X}_\th(t)],    \bar{Y}_\th(t),
                      \bar{Z}_\th(t),\bar{v}(t)  , p_\th(t), q_\th(t), r_\th(t)),   v(t)-\bar v(t)\rangle \bar Q(d\th)dt\]\geq0,
                  \end{aligned}
                 \end{equation*}
                 which implies
                for $v\in V$, $dt\ts d\dbP$-a.e.,
                     \begin{equation*}
                     \begin{aligned}
                    &\int_\Th \langle\pa_vH_\th(t,\bar{X}_\th(t), \dbE[\bar{X}_\th(t)], \bar{Y}_\th(t),
                      \bar{Z}_\th(t), \bar{v}(t),
                   p_\th(t), q_\th(t), r_\th(t)),   v(t)-\bar v(t)\rangle \bar Q(d\th)\geq0.
                    \end{aligned}
                      \end{equation*}
                  \end{proof}

                   Then we turn to the sufficient condition of the optimal control.
                   \begin{theorem}\label{th 5.2}\rm
                  Under the assumptions and the settings in Theorem \ref{th 5.1}, if
                   $H_\th$ is convex in $(x,x',y,z,v)$ and continuous in $t$, $\phi_\th$ is convex in $x$, $\g_\th$ is
                   convex in $y$ and
                    $\F_\th$ is concave in $(x,x')$,
                   and there exists an admissible control $\bar{v}(\cd)$ and a
                  probability measure $\bar Q\in \cQ^{\bar v}$ such that for $v\in V$, $dt\ts d\dbP$-a.e.,
                    \begin{equation*}
                     \begin{aligned}
                    &\int_\Th \langle\pa_vH_\th(t,\bar{X}_\th(t),  \dbE[\bar{X}_\th(t)],   \bar{Y}_\th(t),
                      \bar{Z}_\th(t),\bar{v}(t),
                   p_\th(t), q_\th(t), r_\th(t)),   v-\bar v(t)\rangle \bar Q(d\th) \geq0,
                    \end{aligned}
                      \end{equation*}
                  $\bar v(\cd)$ is an optimal control.
                   \end{theorem}

                   \begin{proof}
                   For $\th\in\Th$ and $v(\cd)\in \cV_{ad}$, by $(X^v_\th,Y^v_\th,Z^v_\th)$ we denote the solution to the
                   equation (\ref{1.1}). Set
                      \begin{equation*}
                     \begin{aligned}
                    (\a_\th^\l,\b_\th^\l,\zeta_\th^\l)=(X_\th^\l-\bar{X}_\th,Y_\th^\l-\bar{Y}_\th,Z_\th^\l-\bar{Z}_\th).
                    \end{aligned}
                      \end{equation*}
                    Then
                    \begin{equation*}
                    \left\{
                     \begin{aligned}
                     \a_\th^\l(t) & =\int_0^t(\pa_xb_\th(s)\a_\th^\l(s)+\pa_{x'}b_\th(s)\dbE[\a_\th^\l(s)]+S_\th(s))ds\\
                                  &\ +\sum_{i=1}^d \int_0^t(\pa_x\si^i_\th(s)\a_\th^\l(s)+J^i_\th(s))dW^i(s),\\
                     \b_\th^\l(t) &=K_\th(T)+\int_t^T(\pa_xf_\th(s)^\top\a_\th^\l(s)+\pa_{x'}f_\th(s)^\top\dbE[\a_\th^\l(s)]+\pa_yf_\th(s)\b_\th^\l(s)\\
                                  &\q  +\pa_zf_\th(s)\zeta_\th^\l(s)+L_\th(s))ds-\int_t^T\zeta_\th(s)dW(s),
                    \end{aligned}
                    \right.
                      \end{equation*}
                   where
                   \begin{equation*}
                     \begin{aligned}
                    S_\th(s)&=b_\th(s, X_\th^v(s),\dbE[X_\th^v(s)],v(s))-b_\th(s)-\pa_xb_\th(s)\a_\th^\l(s)
                              -\pa_{x'}b_\th(s)\dbE[\a_\th^\l(s)],\\
                     J^i_\th(s)&=\si^i_\th(s, X_\th^v(s),v(s))-\si^i_\th(s)-\pa_x\si^i_\th(s)\a_\th^\l(s),\\
                      K_\th(T) & =\F_\th(X_\th^v(T),\dbE[X_\th^v(T)])- \F_\th(\bar{X}_\th(T),\dbE[\bar{X}_\th(T)]),\\
                      L_\th(s) & = f_\th(s, X_\th^v(s),\dbE[X_\th^v(s)],Y_\th^v(s),Z_\th^v(s), v(s))-f_\th(s)-\pa_xf_\th(s)\a_\th^\l(s)\\
                               & \q-\pa_{x'}f_\th(s)\dbE[\a_\th^\l(s)] -\pa_yf_\th(s)\b_\th^\l(s)-\pa_zf_\th(s)\zeta_\th^\l(s)
.
                    \end{aligned}
                      \end{equation*}
                   Applying Ito's formula to $\a_\th^\l(\cd)^\top q_\th(\cd) +\b_\th^\l(\cd)  p_\th(\cd) $ and then integrating from $0$ to $T$, we have

                    \begin{equation*}
                     \begin{aligned}
                     &\dbE[\pa_x\phi_\th(\bar X_\th(T))^\top\a_\th^\l(T)+\pa_y\g_\th(\bar{Y}_\th(0))\b_\th^\l(0)]\\
                     &= -\dbE\[p_\th(T)(\F_\th(X_\th^v(T),\dbE[X_\th^v(T)])- \F_\th(\bar{X}_\th(T),\dbE[\bar{X}_\th(T)]) ) \\
                     &  \q -p_\th(T)\pa_x\F(T)\a_\th^\l(T)-\dbE[p_\th(T)\pa_{x'}\F(T)]\a_\th^\l(T)\]\\
                     &\q +\dbE\[\int_0^T\(q_\th(t)S_\th(t)+\sum_{i=1}^dr^{i}_\th(t)J^i_\th(t)-p_\th(t)L_\th(t)\)dt\].
                    \end{aligned}
                      \end{equation*}
                   Since $\F_\th$ is concave in $(x,x')$ and the  Hamiltonian $H_\th$ is convex in $(x,x',y,z,v)$,  it yields that
                    \begin{equation*}
                     \begin{aligned}
                     &\dbE[\pa_x\phi_\th(\bar X_\th(T))^\top\a_\th^\l(T)+\pa_y\g_\th(\bar{Y}_\th(0))\b_\th^\l(0)]\\
                     & \geq\dbE\[\int_0^T\langle\pa_vH_\th(t,\bar{X}_\th(t), \dbE[\bar{X}_\th(t)], \bar{Y}_\th(t),
                      \bar{Z}_\th(t), \bar{v}(t), p_\th(t), q_\th(t), r_\th(t)),   v^\l(t)-\bar v(t)\rangle  dt\].
                    \end{aligned}
                      \end{equation*}
                   Finally, recall the fact that $\bar{Q}\in \cQ^{\bar{v}}$ and $\phi_\th,\g_\th $ are convex in $x,y$, respectively,
                   it follows from Sion's minimax theorem that
                   \begin{equation*}
                     \begin{aligned}
                     &J(v^\l(\cd))-J(v(\cd))\\
                     &=\int_{\Th}   \dbE\[  \phi_\th(X_\th^{\l}(T))-\phi_\th(\bar{X}_\th (T))
                     +\g_\th(Y^{\l}_\th(0))-\g_\th(\bar{Y}_\th(0))\] \bar{Q}(d\th)\\
                     &\geq\int_{\Th}  \dbE[\pa_x\phi_\th(\bar X_\th(T))^\top\a_\th^\l(T)+\pa_y\g_\th(\bar{Y}_\th(0))\b_\th^\l(0)]\bar{Q}(d\th)\\
                                & \geq\int_{\Th}  \dbE\[\int_0^T\langle\pa_vH_\th(t,\bar{X}_\th(t), \dbE[\bar{X}_\th(t)],   \bar{Y}_\th(t),
                      \bar{Z}_\th(t),\bar{v}(t),p_\th(t), q_\th(t), r_\th(t)),   v^\l(t)-\bar v(t)\rangle  dt\]  \bar{Q}(d\th)\\
                  &\geq \dbE\[\int_0^T\int_{\Th} \langle\pa_vH_\th(t,\bar{X}_\th(t),  \dbE[\bar{X}_\th(t)],   \bar{Y}_\th(t),
                      \bar{Z}_\th(t),\bar{v}(t),p_\th(t), q_\th(t), r_\th(t)),   v^\l(t)-\bar v(t)\rangle\bar{Q}(d\th)  dt\]\geq0.
                    \end{aligned}
                      \end{equation*}
                The proof is complete.

                   \end{proof}

\end{document}